\documentclass[11pt]{article}
\usepackage[utf8]{inputenx}
\usepackage{graphicx}
\usepackage{amsmath}
\usepackage{amsfonts}
\usepackage[english]{babel}
\usepackage{amsthm}
\usepackage{latexsym}
\usepackage{xcolor}
\usepackage{comment}
\usepackage{hyperref}
\usepackage{subfig} 
\usepackage{soul}

\numberwithin{equation}{section}
\newtheorem{theorem}{Theorem}[section]
\newtheorem{assumption}{Assumption}
\newtheorem{corollary}[theorem]{Corollary}
\newtheorem{definition}[theorem]{Definition}
\newtheorem{lemma}[theorem]{Lemma}
\newtheorem{proposition}[theorem]{Proposition}
\newtheorem{remark}[theorem]{Remark}

\newtheorem{algorithm}{Algorithm}
\def\neweq#1{\begin{equation}\label{#1}}
\def\endeq{\end{equation}}

\input texdraw

\newcommand{\R}{\mathbb{R}}

\renewcommand{\div}{{\rm div}}
\newcommand{\om}{\Omega}
\newcommand{\omdel}{\Omega_{\varrho}}
\newcommand{\omdd}{\Omega\setminus{D}}
\newcommand{\omd}{\Omega_D}

\newcommand{\eps}{\epsilon}

\DeclareMathOperator*{\argmin}{arg\,min}
\newcommand{\proj}{\ensuremath{\mathrm{proj}}} 
\newcommand{\Jade}{J_{\delta,\varepsilon}}

\begin{document}

\title{On the reconstruction of cavities in a nonlinear model arising from cardiac electrophysiology }

\author{Elena Beretta \footnote{Division of Science, NYU Abu Dhabi, \texttt{eb147@nyu.edu}}\ ,
		M. Cristina Cerutti \footnote{Dipartimento di Matematica, Politecnico di Milano , \texttt{cristina.cerutti@polimi.it}}\ ,
	Dario Pierotti, \footnote{Dipartimento di Matematica, Politecnico di Milano , \texttt{dario.pierotti@polimi.it}}\,
	Luca Ratti\footnote{Machine Learning Genoa Center, Department of Mathematics, University of Genoa,  \texttt{luca.ratti@unige.it}}}

\date{ }

\maketitle

\begin{abstract}
In this paper we deal with the problem of determining perfectly insulating regions (cavities) from one boundary measurement in a nonlinear elliptic equation arising from cardiac electrophysiology. 
Based on the results obtained in \cite{BCP} 
we propose a new reconstruction algorithm 
based on $\Gamma$-convergence.
 The relevance and applicability of this approach is then shown through several numerical experiments. 
\end{abstract}
\vskip.15cm

\vskip.15cm


\section{Introduction}

In this paper we tackle the inverse problem of reconstructing a cavity $D$ within a planar domain $\Omega$ taking advantage of boundary measurements of the solution of the following boundary value problem:
\begin{equation}\label{cavityproblem}
    \left\{
    \begin{aligned}
        -\Delta u + u^3 &= f \qquad & \text{in } \Omega \setminus D \\
        \frac{\partial u}{\partial\mathbf n} &= 0 & \text{on } \partial \Omega \cup \partial D.
    \end{aligned}
    \right.
\end{equation}
The investigation of this problem is mainly motivated by the mathematical modelling of the electrical activity of the heart regarding, in particular, the detection of ischemic regions from boundary measurements of the transmembrane potential, \cite{BCP}. These regions are composed of non-excitable tissue, that can be modeled as an electrical insulator (cavity) \cite{Perez}, \cite{Relan},\cite{Fronteraetal}. Identification of ischemic regions and their shape is fundamental to perform successful radiofrequency ablation for the prevention of tachycardias and of more serious heart disease. In the steady-state case the transmembrane potential in the presence of an ischemic region satisfies exactly Problem (\ref{cavityproblem}). Hence, mathematically, the inverse problem boils down in determining a cavity $D$ from boundary measurements of the solution $u$. In \cite{BCP} part of the authors analyzed the well-posedness of (\ref{cavityproblem}) and uniqueness of the inverse problem under minimal regularity assumptions on the unknown cavity. More precisely, they proved that one measurement of the potential $u$ on an open arc of $\partial\Omega$ is enough to detect uniquely a finite union of disjoint, compact, simply connected sets with Lipschitz boundary.
The inverse problem is highly nonlinear and severly ill-posed since, as for the linear conductivity problem, even within a class of smooth cavities only a very weak logarithmic-type continuous dependence on data is expected to hold , see \cite{AMR}, \cite{ABRV}. 

In \cite{BRV} the authors analyzed the mathematical model in the case of conductivity inhomogeneities of arbitrary shape and size in the two-dimensional setting. In particular, the issue of reconstructing the inhomogeneity from boundary measurements was addressed. The strategy used in \cite{BRV} for the reconstruction from few data was based on the minimization of a quadratic mismatch functional with a perimeter penalization term. In order to derive a more manageable problem, the perimeter functional was relaxed using a phase-field approach justified by showing the $\Gamma$-convergence of the relaxed functional to the functional with perimeter penalization.  In recent years this kind of approach has been successfully implemented in inverse boundary value problems for partial differential equations and systems,  see for example \cite{BRV},\cite{BZ}, \cite{R}, \cite{LY}, \cite{DES}.\\
Here, we use a similar approach starting from the minimization of the following quadratic boundary misfit functional with a Tikhonov regularization penalizing the perimeter of the set $D$: 
\begin{equation}
\label{perfunct_intro}J(D) = \frac{1}{2} \int_{\Sigma}(u(D) - u_{meas})^2d\sigma+\alpha \textrm{Per}(D)
\end{equation}
where $\alpha>0$ represents the regularization parameter, $u_{meas}$ the measurements corresponding to some solution of (\ref{cavityproblem}). Assuming uniform Lipschitz regularity of the cavity $D$, we prove continuity of solutions to (\ref{cavityproblem})  with respect to $D$ in the Hausdorff metric and, as a consequence, the existence of minima in the class of Lipschitz cavities $D$, showing the stability of the functional with respect to noisy data and the convergence of minimizers as $\alpha\rightarrow 0$ to the solution of the inverse problem.
\\
In the linear counterpart of the problem, it is natural to interpret cavities as perfectly insulating inclusions, namely regions in which the conductivity of the medium is vanishing. This scenario (together with the case of perfectly conducting inclusions, where the conductivity goes to infinity) is usually referred to as \textit{extreme} conductivity inclusions. For this reason, it is natural to interpret the cavity problem under exam as the limit case of the inclusion detection, and hence to approximate it by means of inclusion detection problems associated with very low conductivities $\delta$. This entails the introduction of an approximation of the forward problem \eqref{cavityproblem}, leading to a solution map $u_\delta$ and to the corresponding functional $J_{\delta}$ and minimization problem (see \eqref{mindelta}). 
Since the functional is not differentiable and its minimization is conducted in a non-convex space, we propose as in \cite{BRV} a Modica-Mortola relaxation of the functional $J_{\delta}$ via a family of smooth functionals $J_{\varepsilon,\delta}$ defined on a suitable subset of $H^1(\Omega)$ to guarantee $\Gamma-$ convergence as $\varepsilon\rightarrow 0$ and as $\delta\rightarrow 0$ to the functional $J$. \\ 
This theoretical convergence result motivates the choice to approximate the original regularized problem \eqref{perfunct_intro} by minimizing the functional $J_{\delta,\varepsilon}$ for fixed, small values of $\delta$ and $\varepsilon$. The Fr\'echet differentiability of such functionals ultimately suggests to employ a first-order optimization method to iteratively converge to a critical point, satisfying (necessary) optimality conditions. As further motivated in Section \ref{sec:numerics}, we can sequentially perform the minimization of $J_{\delta,\varepsilon}$ for reducing values of $\varepsilon$ and $\delta$ to obtain a candidate regularized solution of the original cavity detection problem. \\
Nevertheless, there is a gap between the theoretical results and the numerical implementation: in particular, unlike the conductivity case, the phase-field relaxation via $J_{\delta,\varepsilon}$ is not able to mitigate the non-convexity of the original problem. Indeed, as explained above, we must assume that the cavity $D$ is of Lipschitz class. To guarantee such a regularity for the minimizers of the functionals $J_\delta$, and to ensure the $\Gamma-$convergence as $\varepsilon \rightarrow 0$, we are forced to set the minimization of $J_{\delta,\varepsilon}$ in a suitable non-convex subset $\mathcal{K}_{\eta}$ of $H^1(\Omega;[0,1])$. This allows, on the one hand, for a complete and thorough theoretical analysis of the relaxation strategy, but on the other one, it still makes it impossible to minimize $J_{\delta,\varepsilon}$ by means of standard gradient-based schemes. However, numerical evidence shows that it is possible to perform such a minimization on the whole space $H^1(\Omega;[0,1])$ and still have convergence to a function satisfying the desired additional regularity.
\par
From a numerical standpoint, we can compare our strategy with other existing approaches in the literature related with the linear counterpart of the problem, namely the cavity detection in the linear conductivity problem. In such a context, phase-field techniques have been studied for the reconstruction of cavities (and cracks) in the conductivity case in \cite{ring2010reconstruction} and in the elasticity case in \cite{ABCRV} and \cite{A}. \\
Among the several alternative strategies, we can perform a main distinction between algorithms which have been originally developed for inclusion detection and later extended to the cavity case, and algorithms specifically suited for the reconstruction of cavities. 
\\
Regarding the first family, we can trace back to the first algorithm, introduced by Friedman and Vogelius in \cite{friedman1989} for the detection of arbitrarily small (extreme) conductivity inclusions. It is based on the asymptotic expansion of a suitable mismatch functional, and it has been further developed, by means of polarization tensors, by Ammari and Kang in \cite{ammari2007polarization}. Subsequently, many other techniques originally designed for inclusion detection have been extended to the cavity case. For example, the enclosure method, allowing for the reconstruction of the convex hull of inclusions in electrical impedance tomography, has been formulated in the cavity case 
in \cite{ikehata2002numerical}, whereas the factorization method, developed by Br\"uhl and Hanke, has been investigated for the cavity problem in \cite{hanke2003recent}, also comparing it to the MUSIC algorithm. Analogously, the level set method, allowing for the reconstruction an inclusion as a level curve of a suitable function which is iteratively updated, has been successfully and efficiently implemented 
in \cite{burger2003levenberg}. Recently, also the monotonicity method, which exploits the monotonicity of the Dirichlet-to-Neumann map to define an iterative reconstruction algorithm, has been studied in the presence of extreme inclusions, 
see \cite{candiani2020monotonicity}. 
\\
Among the second family of algorithms, namely the ones that are innatily suited for extreme inclusions detection in the linear conductivity equation, we recall the the method of fundamental solutions 
(see \cite{borman2009method}), the algorithm by Kress and Rundell involving nonlinear boundary integral equations (see \cite{kress2005nonlinear}) and the conformal mapping technique (see \cite{kress2012inverse} and \cite{munnier2017conformal}).

\par
The remaining part of the paper is structured as follows: in Section \ref{sec:notation} we set the notation and introduce the main assumptions regarding the forward problem and the class of cavities we aim at reconstructing. Section \ref{sec:direct} is devoted to the analysis of the forward problem \eqref{cavityproblem}, both recalling the well-posedness results from \cite{BCP} and proving a novel result about the continuous dependence of boundary measurements with respect to  cavity perturbations. Section \ref{sec:recon} outlines our approach to the reconstruction of cavities, studying the regularization properties of \eqref{perfunct_intro} and thoroughly describing the phase-field relaxation. It also contains the main theorems of the paper, namely the $\Gamma-$convergence results for the relaxed functionals as $\delta$ and $\varepsilon$ go to $0$. Finally, Section \ref{sec:numerics} provides a numerical counterpart of the proposed strategy, formulating and discussing two optimization algorithms, and in Section \ref{sec:results} we report the results of some numerical experiments, assessing the effectiveness of such algorithms. 

\section{Notation and main assumptions}
\label{sec:notation}

We consider the following inhomogeneous Neumann problem 
\begin{equation}
\label{probcav}
\left\{
  \begin{array}{ll}
    -\Delta u+u^3=f, & \hbox{in $\omdd$} \\
    \displaystyle{\frac{\partial u}{\partial\mathbf{n}}}=0, & \hbox{on $\partial(\omdd)$}.
  \end{array}
\right.
\end{equation}

where $\mathbf{n}$ is the outer unit normal to $\Omega\backslash D$.
In what follows we will use the notation 

$$\omd=\omdd$$

Let us first  recall the definition of Lipschitz (or $C^{0,1}$) regularity.
\begin{definition}
[Lipschitz or $C^{0,1}$ regularity]\ \\\label{lipclass}
Let $\Omega$ be a bounded domain in $\mathbb{R}^2$.  We say that a portion $S$ of $\partial \Omega$ is of class $C^{0,1}$ with constants $r_0$ and $L_0$, if for any ${P}\in S$, there exists a rigid transformation of coordinates under which we have ${P}={0}$ and
	\begin{equation*}
		\mathring{\Omega}\cap B_{r_0}({0})=\{{(x_1,x_2)}\in B_{r_0}({0})\, :\, x_2>\psi(x_1)\},
	\end{equation*}
where ${\psi}$ is a $C^{0,1}$ function on $(-r_0,r_0)\subset \mathbb{R}$ such that
	\begin{equation*}
		\begin{aligned}
			{\psi}({0})&=0,\\
			\|{\psi}\|_{C^{0,1}(B_{r_0}({0}))}&\leq L_0.
		\end{aligned}
	\end{equation*}

We say that e domain $\Omega$ is Lipschitz (or a $C^{0,1}$) with constants $r_0$, $L_0$, if $\partial\Omega\in C^{0,1}$.
\end{definition}
We can now state our main set of assumptions. 
\begin{assumption}\label{as:1}
 $\om\subset \R^2$  is a bounded Lipschitz domain with constants $r_0, L_0$.
\end{assumption}

\begin{assumption}\label{as:2} $\Sigma \subset \partial \Omega$, an open arc, is the portion of boundary which is accessible for measurement.
\end{assumption}

\begin{assumption}\label{as:3}
The cavity $D$ is the union of at most $M$ disjoint, compact, simply connected sets and has Lipschitz boundary, i.e  $D\in\mathcal{D}$ defined by

$$
\mathcal{D}=\{\exists N\leq M \;|\; D=\cup_{j=1}^N D_j\subset \Omega\,, \partial D\in C^{0,1} \textrm{ with constants } r_0,L_0\}
$$
where, $\forall \;1\leq j\leq N$, $D_j$ is compact and simply connected; moreover we assume
${\rm dist}(D_j,D_i)\geq d_0\,\,\,\forall i\neq j$ and ${\rm dist}(D,\partial\Omega)\geq 2 d_0>0.$

\end{assumption}
\begin{assumption}\label{as:4} The source term $f$ in \eqref{probcav} satisfies
\begin{equation}
 f \in L^{\infty}(\Omega),\,\, f\geq 0, \,\,\,\mathrm{supp}(f)\subset\Omega_{d_0}=\{x\in\Omega: dist(x,\partial\Omega)\leq d_0\}.
  \end{equation}
  \end{assumption}
\vskip 2truemm
We will denote by
$$\textrm{Per}(D)=\operatorname{TV}(\chi_D,\Omega)=\sup\left\{\int_{\Omega}\chi_D(x)\textrm{div}\varphi(x)dx:\varphi\in C_c^1(\Omega,\mathbb{R}^2), |\varphi|\leq 1\right\}.$$
In particular, if $D\in\mathcal{D}$ we have that
$$
\textrm{Per}(D)=\mathcal{H}^1(\partial D)<+\infty
$$
where $\mathcal{H}^1(\partial D)$ is the one-dimensional Hausdorff measure of $\partial D$.

\vskip 5truemm

In the sequel $A\triangle B:=(A\backslash B)\cup(B\backslash A)$ will indicate the symmetric difference of the two sets $A$ and $B$. Finally, let us recall the definition of the Hausdorff distance between two sets $A$ and $B$:
\[d_H(A,B)=\max\{\sup_{x\in A}\inf_{y\in B}\textrm{dist}
(x,y),\sup_{y\in B}\inf_{x\in A}\textrm{dist}(y,x)\}.\]
\begin{remark}
Throughout the paper, for the sake of brevity, we will denote with $v$ the indicator function of some set $D\subset \Omega$ and we will use $\textrm{Per}(D)$ or $\operatorname{TV}(v)$ depending on the situation.
\end{remark}
We will use several times throughout the paper the following compactness result
\begin{proposition}
\label{compactD}
$\mathcal{D}$ is compact with respect to the Hausdorff topology.
\end{proposition}
\begin{proof} Let us first consider the case $M=1$ i.e.  $D\in \mathcal{D}$ is compact and simply connected. Let  $\{D_k\}_{k=1}^{\infty}\subset\mathcal{D}$ be a sequence of sets in $\Omega$. Then by Blaschke’s Selection Theorem (see for example Theorem 3.1 in \cite{dalMT}) there exists a subsequence that we still indicate by $\{D_k\}_{k=1}^{\infty}$ converging in the Hausdorff metric to a compact set $D$. Furthermore, as a consequence of Theorem 2.4.7, Remark 2.4.8 and Theorem 2.4.10 of \cite{HP}, $\{\partial D_k\}_{k=1}^{\infty}$ converges in the Hausdorff metric to  $\partial D$  and  $\partial D$ is Lipschitz with constants $r_0,L_0$ and is connected, which implies that $D$ is also simply connected. So, $D\in \mathcal{D}$ which concludes the proof. 

If $M>1$ and $D_k\rightarrow D=\cup_{j=1}^N D_{j}$, where, $\forall j$, $D_j$ is simply connected and its boundary is Lipschitz with constants $r_0,L_0$. Because of the uniformity of the Lipschitz property, for sufficiently large $k$ we have that $D_k=\cup_{j=1}^N D_{j,k}$, i.e. $D_k$ has the same number of disjoint connected components ad $D$. Moreover, for any fixed $j$, $D_{j,k}\rightarrow D_j$, possibly up to a subsequence. Finally, by the definition of Hausdorff distance we conclude that $d(D_j,D_i)\geq d_0$ for any $i\neq j.$

\end{proof}

\section{Analysis of the direct problem}
\label{sec:direct}

\subsection{Well posedness and main estimates}

We first recall a well posedness result for problem (\ref{probcav}) proved in \cite{BCP} (in a more general setting) together with some estimates on the solution which will be useful in the subsequent discussion. Note that, by assumptions $1$, $3$,
the domains $\om_D$  have Lipschitz  boundaries for any $D\in \mathcal{D}$. Then we have:

\begin{theorem}
\label{exist}
Suppose that Assumptions $1-4$ hold. Then problem \eqref{probcav} has a unique solution $u\in H^1(\omd)$.
Furthermore, the following bounds hold:

\begin{equation}
\label{apriori}
\|u\|_{H^1(\omd)}
\le C\big (\|f\|_{(H^1)'}+\|f\|^{1/3}_{(H^1)'}\big )
\end{equation}

\begin{equation}\label{boundsu}
0\leq u(x) \leq\left(ess\,\sup_{\omd}f\right)^{1/3}\quad\quad \mathrm{a.e.}\quad x\in\omd\,.
\end{equation}

where $C=\max\{1,|\omd|^{1/3}$\} and $(H^1)'=H^1(\omd)'$ is the dual space of the Sobolev space $H^1(\omd)$.
\end{theorem}

The proof follows by suitable Sobolev estimates and by the maximum principle, see \cite{BCP} Proposition 3.4 and Theorem 3.5.

\subsection{Continuity properties of the solutions with respect to $D$}

Let us consider the weak formulation of  Problem \eqref{probcav}

\begin{equation}
\label{weakfucav1}
\int_{\omd} \nabla u\cdot\nabla\phi + \int_{\omd}u^3\phi=\int_{\omd} f\phi, \ \ \ \forall\phi\in H^1(\omd).
\end{equation}

\smallskip
By Theorem \ref{exist}, there is a unique solution $u_D\in H^1(\omd)$ of \eqref{weakfucav1} which is uniformly bounded in  $H^1(\omd)$ and in $L^{\infty}(\omd)$ by constants depending only on $f$ (for a given $\om$).

\smallskip
In this section, we will prove the continuity of the trace map $D\mapsto u_D\big |_{\Sigma}$
for domains in the  class $\mathcal{D}$ defined in assumption 3; more precisely :

\smallskip
 let $D_n\in \mathcal{D}$ be a sequence of sets converging to $D$ in the metric defined by the Hausdorff distance $d_H$ and let
$u_n:=u_{D_n}$, $u:=u_D$. Then

\begin{equation}
\label{convdom}
\lim_{n\to \infty}\int_{\Sigma}|u_n-u|^2
=0\,.
\end{equation}

\smallskip
The proof of our claim will require some intermediate steps.

To begin with, by known results on approximation of bounded Lipschitz domains (see e.g. \cite{Ver} theorem $1.12$)  one can construct, for any $\varrho>0$,  a subset $D_{\varrho}$ such that $\Omega_{\varrho}:=\Omega\backslash D_{\varrho}$ satisfies the following properties

\begin{enumerate}
  \item $\Omega_{d_0}\subset\omdel\subset\subset\om_D$ and $\partial\omdel$ is $C^{0,1}$ (actually smooth) 
  \item  $\big | \om_D\setminus \omdel\big |<\varrho$\,.
\end{enumerate}

Then, by the convergence of $D_n$ to $D$ in the Hausdorff metric (see the proof of Proposition \ref{compactD} above) there exists a positive integer $n_{\varrho}$ such that
$\omdel\subset\subset\om_{D_n}$ for every $n>n_{\varrho}$. 

Note that
\begin{equation}
\label{omenodn}
\big | \om_{D_n}\setminus \omdel \big |\le \big | \om_D\setminus \omdel\big |+
\big |D\setminus D_n  \big |<\varrho+o(1)\,,
\end{equation}
for $n\to\infty$.

\smallskip
Then we have
\begin{theorem}
\label{contdom}
Let $u$, $u_n$, $\omdel$, $n_{\varrho}$ be defined as above. Then, for any $\epsilon>0$ there exists $\varrho(\epsilon)>0$ such that, for every $\varrho<\varrho(\epsilon)$ and $n>n_{\varrho}$,
\begin{equation}
\label{convdomeps}
\|u_n-u\|_{H^1(\omdel)}<\epsilon\,.
\end{equation}
\end{theorem}

\begin{proof}
Since $u$, $u_n$, solve \eqref{weakfucav1} respectively in $\omd$ and in $\Omega_{D_n}$ and recalling that supp$\,f\subset \Omega_{d_0}\subset\omdel$, we have
\begin{equation}
\nonumber
\int_{\omd} \nabla u\cdot\nabla\phi + u^3\phi=\int_{\Omega_{D_n}} \nabla u_n\cdot\nabla\phi + u_n^3\phi\,.
\end{equation}
$\forall\phi\in H^1(\om)\,$ (note that by our assumptions on the domains, any $\phi\in H^1(\omd)$ or in $H^1(\Omega_{D_n})$ is the restriction of a function in $H^1(\om)$).

\smallskip
By the decompositions $\omd=\omdel\,\cup\,\big (\omd\setminus\omdel \big )$, $\,\,\Omega_{D_n}=\omdel\,\cup\, \big (\om_{D_n}\setminus\omdel\big )$, we have

\begin{equation}
\nonumber
\int_{\omdel} \nabla u\cdot\nabla\phi + u^3\phi
+\int_{\omd\setminus\omdel} \nabla u\cdot\nabla\phi + u^3\phi=
\end{equation}

\begin{equation}
\nonumber
\int_{\omdel} \nabla u_n\cdot\nabla\phi + u_n^3\phi+\int_{\om_{D_n}\setminus\omdel} \nabla u_n\cdot\nabla\phi + u_n^3\phi\,.
\end{equation}

By rearranging terms:

\begin{equation}
\nonumber
\int_{\omdel}  \nabla (u-u_n) \cdot\nabla\phi + (u^3-u^3_n)\phi=
\end{equation}

\begin{equation}
\label{weakformdiffer}
-\int_{\omd\setminus\omdel} \nabla u\cdot\nabla\phi + u^3\phi
+\int_{\om_{D_n}\setminus\omdel} \nabla u_n\cdot\nabla\phi + u_n^3\phi\,.
\end{equation}

Let $\phi_{\varrho,n}\in H^1(\om)$ be a function satisfying

\begin{equation}
\nonumber
\phi_{\varrho,n}\,\big |_{\omdel}=(u-u_n)\,\big |_{\omdel}\,.
\end{equation}

The existence of $\phi_{\varrho,n}$ for every $\varrho$ (and $n$) follows by the extension property which holds for the Lipschitz domain $\omdel$.
Moreover, by the uniform bounds on $u$, $u_n$ in $\Omega_{\varrho}$ and by the continuity of the extension operator, we readily get

\begin{equation}
\label{stimphindel}
\|\phi_{\varrho,n}\|_{H^1(\om)}\le C\,,
\end{equation}

where the constant $C$  depends only on $\omdel$ and $f$. Actually, by properties $1$ and $2$ above and since $\Omega_D$ is Lipschitz we can take $C$ independent of $\varrho$.

\smallskip
By choosing $\phi=\phi_{\varrho,n}$ in \eqref{weakformdiffer} we obtain

\begin{equation}
\nonumber
\int_{\omdel} \nabla (u-u_n) \cdot\nabla (u-u_n) + (u-u_n)^2(u^2+u u_n+u_n^2)=
\end{equation}

\begin{equation}
\label{weakformdiff}
-\int_{\omd\setminus\omdel} \nabla u\cdot\nabla\phi_{\varrho,n} + u^3\phi_{\varrho,n}
+\int_{\om_{D_n}\setminus\omdel} \nabla u_n\cdot\nabla\phi_{\varrho,n} + u_n^3\phi_{\varrho,n}\,.
\end{equation}

We now estimate the integrals at the right hand side. First, since $u\in H^1(\omd)\cap L^{\infty}(\omd)$ we get by \eqref{stimphindel} and by
Holder inequality

\begin{equation}
\nonumber
\Big |\int_{\omd\setminus\omdel} \nabla u\cdot\nabla\phi_{\varrho,n} + u^3\phi_{\varrho,n}\Big |\le
C_1\Big (\int_{\omd\setminus\omdel} |\nabla u|^2\Big )^{1/2}+
C_2\big |\omd\setminus\omdel  \big |^{1/2}\,,
\end{equation}

with $C_1$, $C_2$ independent of $\varrho$ and $n$. By Property $2$ and by the integrability of $|\nabla u|^2$, we can now write

\begin{equation}
\label{estrhs1b}
\Big |\int_{\omd\setminus\omdel} \nabla u\cdot\nabla\phi_{\varrho,n} + u^3\phi_{\varrho,n}\Big |\le
C_1 \,o(1)+
C_2\,\varrho^{1/2}\,,
\end{equation}

for $\varrho\to 0$. By similar estimates of the second term at the right hand side of \eqref{weakformdiff} and taking into account \eqref{omenodn} we obtain

\begin{equation}
\nonumber
\Big |\int_{\om_{ D_n}\setminus\omdel} \nabla u_n\cdot\nabla\phi_{\varrho,n} + u_n^3\phi_{\varrho,n}\Big |\le
C_1\Big (\int_{\om_{ D_n}\setminus\omdel} |\nabla u_n|^2\Big )^{1/2}+
C_2\,\big (\varrho+o(1)\big )^{1/2}\,.
\end{equation}

\emph{Claim:}

the ${u_n}'s$ can be extended to $\om$ in such a way that the sequence $|\nabla u_n|^2$ is \emph{uniformly integrable} in $\om$.

\smallskip
\begin{proof}
The function $u_n$ is a weak solution of the Neumann problem
\begin{equation}
\nonumber
\left\{
  \begin{array}{ll}
    -\Delta u_n=f-u_n^3, & \hbox{in $\om_{D_n}$} \\
    \displaystyle{\frac{\partial u}{\partial\mathbf{n}}}=0, & \hbox{on $\partial\om_{D_n}$},
  \end{array}
\right.
\end{equation}

By the estimates of the previous section the right hand side of the above equation satisfies
\begin{equation}
\nonumber
\mathrm{ess}\,\sup_{\om_{D_n}} |f-u_n^3|\le C\,,
\end{equation}

with $C$ independent of $n$. Then, known regularity results for the Neumann problem in Lipschitz domains \cite{JK}, \cite{Cost} imply that
$u_n\in H^{3/2}(\Omega_{D_n})$ with uniformly bounded norm. Moreover, $u_n$ has an extension (still denoted by $u_n$)
to $\om$ satisfying

\begin{equation}
\nonumber
\|u_n\|_{ H^{3/2}(\Omega)}\le C\,
\end{equation}

(see \cite{Grisvard} Theorem 1.4.3.1). Hence, by Sobolev imbeddings $\{u_n\}_{n\in \mathbb{N}}$ is a relatively compact subset of $H^1(\om)$; in particular, $|\nabla u_n|^2$ is relatively compact in $L^1(\om)$ and therefore is uniformly integrable.
\end{proof}

The above implies that estimate \eqref{estrhs1b} holds for \emph{both terms} on the right hand side of \eqref{weakformdiff}.

\smallskip
Hence, by taking $\varrho<\varrho(\eps)$ small enough and $n>n_{\varrho}$ large, we have

\begin{equation}
\nonumber
\int_{\omdel} \nabla (u-u_n) \cdot\nabla (u-u_n) + (u-u_n)^2(u^2+u u_n+u_n^2)\le \epsilon^2\,.
\end{equation}

Finally, since $u^2+u u_n+u_n^2\ge \frac{3}{4} u^2$ and  $\|u\|_{L^{\infty}(\omdel)}>0$, by Poincar\'{e} inequality (Theorem A.1 in \cite{BCMP}) we get

\begin{equation}
\nonumber
\|u-u_n\|^2_{H^{1}(\omdel)}\le  C\,\epsilon^2\,,
\end{equation}

for some constants $C$ independent of $\varrho$, $n$. Then, the result follows by redefining
$C^{1/2}\epsilon\rightarrow\epsilon$.

\end{proof}

We can now prove

\begin{corollary}
\label{limitdom}
Let $\Omega$, $\Omega_{d_0}$ be defined as in Section $2$ and let $D_n\,,D\subset \om$, $n=1,2,...$ such that  ${D_n}\in \mathcal{D}$, and $D_n\rightarrow D$ in the Hausdorff metric. Let $u\in H^1(\omd)$ and $u_n\in H^1(\om_{D_n})$ be the solutions of \eqref{weakfucav1} in $\omd$ and in $\om_{D_n}$ respectively. Then,

\begin{equation}
\label{convdomteo}
\lim_{n\to \infty}\|u_n-u\|_{L^2(\Sigma)}
=0\,.
\end{equation}

\end{corollary}

\begin{proof}
Note that $D\in \mathcal{D}$ by the proof of Proposition \ref{compactD}.
Fix $\epsilon>0$ and let $\varrho<\varrho(\eps)$, $n>n_{\varrho}$ such that \eqref{convdomeps} holds. Then, by standard trace theorems

\begin{equation}
\nonumber
\int_{\Sigma}|u_n-u|^2\le C\,\|u_n-u\|^2_{H^1(\Omega_{d_0})}\le C\|u_n-u\|^2_{H^1(\omdel)}<\epsilon^2\,
\end{equation}

and the corollary follows.

\end{proof}

\smallskip

\begin{remark}
\label{convinmisur}
Let $D_n$, $D$, be as in Corollary \ref{limitdom}, but assume that ${D_n}\rightarrow D\in \mathcal{D}$ in measure in $\om$, that is $|D_n\triangle D|\rightarrow 0$ (see \cite{AFP}, Remark $3.37$).
 Nevertheless, by Proposition \ref{compactD} the sequence ${D_n}$ has compact closure in the Hausdorff topology. Hence, there exists a subsequence $D_{n_k}$ which converges in the Hausdorff metric. Then the subsequence necessarily converges to $D$, since $D$ has (Lipschitz) continuous boundary; hence, the above corollary applies to this subsequence.
\end{remark}
\begin{remark}
Since we are considering the case of domains that are uniformly Lipschitz it is also possible to  derive the continuity of solutions with respect to perturbations of the cavities in the Hausdorff topology in the framework of Mosco convergence, see Theorem 7.2.7 in \cite{BB} and \cite{CD}. 
\end{remark}

\section{Reconstruction of cavities}
\label{sec:recon}

In \cite{BCP} the authors prove uniqueness for the inverse problem, i.e. that assuming it is possible to measure the solution $u$ to \eqref{probcav} on $\Sigma$ the cavity $D$ is uniquely determined. In this section we deal with the problem of reconstructing the cavity presenting a new rigorous algorithm based on $\Gamma$-convergence. We start by formulating a minimization problem for a functional $J$ depending on the cavity $D$ as a variable: let $u=u(D)$ be the unique $H^1(\omd)$ solution of the boundary value problem \eqref{probcav}, and consider

In order to reconstruct the cavity $D$ a natural approach is to minimize a quadratic misfit functional measuring the discrepancy on the boundary between the solution and the data perturbed with a term that penalizes the perimeter of the cavity $D$, i.e.
 
\begin{equation}\label{minper}
\min_{D\in \mathcal{D}} J(D):\, J(D) = \frac{1}{2} \int_{\Sigma}(u(D) - u_{meas})^2d\sigma+\alpha \textrm{Per}(D),
\end{equation}
where $\mathcal{D}$ and $\textrm{Per}(D)$ are specified in Section \ref{sec:notation}, and $\alpha>0$ is referred to as the \textit{regularization parameter}, which balances the contribution of the quadratic mismatch term and the regularization term in $J$. 

In Subsection \ref{subsec:regu}, we show that the minimization problem \eqref{minper} satisfies several desirable properties and can be interpreted as a regularization strategy for the inverse problem of determining $D$ from $u_{meas}$. In Subsection \ref{subsec:delta} we propose an approximation of problem \eqref{minper} by perturbing the solution map $u(D)$ appearing in it. The resulting problem is later relaxed by a phase-field approach presented in Subsection \ref{subsec:eps}. Finally, in Subsection \ref{subsec:gamma}, we prove the convergence of the introduced approximate functionals to the original $J$ in the sense of the $\Gamma$-convergence, which also entails the convergence of the associated minimizers.

\subsection{Regularization properties of the minimization problem}
\label{subsec:regu}

The minimization problem \eqref{minper} can be interpreted as a regularized counterpart of the inverse problem of determining $D$ from $u_{meas}$. In particular, in the following three results, we prove that for every $\alpha >0$ the functional $J$ admits a minimum, that the minimizers are stable with respect to perturbations of the datum $u_{meas}$ and that, if the amount of the noise on the measurements converges to zero and the parameter $\alpha$ is suitably chosen, the solutions of \eqref{minper} converge to the unique solution $D$ of the inverse problem.

\begin{proposition}\label{mincav}
 For every $\alpha >0$ there exists at least one solution to the minimization problem (\ref{minper}).
\end{proposition}
\begin{proof}
Assume $\{D_k\}_{k\geq 0}\in\mathcal{D}$ is a minimizing sequence. Then there exists a positive constant $C$ such that
$$J(D_k)\leq C,\,\,\, \forall k$$
 and in particular
 $$\textrm{Per}(D_k) \leq C,\,\,\forall k.$$
 Then by compactness (see for example \cite{AFP} Theorem 3.39) there exists a set of finite perimeter $D_0$ such that, possibly up to a subsequence,
$$
 |D_k\triangle D_0|\rightarrow 0,\,\, k\rightarrow \infty.
$$
Then by the lower semicontinuity property of the perimeter functional (see \cite{EG} Section 5.2.1, Theorem 1)  it follows that

$$
\textrm{Per}(D_0)\leq \liminf_{k\rightarrow\infty}\textrm{Per}(D_k).
$$
Moreover, by Remark \ref{convinmisur} and Proposition \ref{compactD} we may assume that the sequence also converges in the Hausdorff metric to $D_0\in \mathcal{D}$.
Hence, by Corollary \ref{limitdom} it follows that
\[
  \int_{\Sigma}(u(D_k) - u_{meas})^2d\sigma\rightarrow  \int_{\Sigma}(u(D_0) - u_{meas})^2d\sigma\,\,\textrm{as }k\rightarrow\infty.
  \]
Finally,
\[
J(D_0)\leq \liminf_{k\rightarrow\infty}J(D_k)=\lim_{k\rightarrow\infty}J(D_k)=\inf_{D\in \mathcal{D}}J(D)
\]
and the claim follows.
\end{proof}


\begin{proposition}
The solutions of (\ref{minper}) are stable w.r.t. perturbation of the data $u_{meas}$ i.e. if $\{u_k\}\subset L^2(\Sigma)\rightarrow u_{meas}$ in $L^2(\Sigma)$ as $k\rightarrow \infty$ then the solutions $D_k$ of (\ref{minper}) with datum $u_k$ are such that, up to subsequences,
$$
d_H(D_k,\bar D)\rightarrow 0,\,\, \textrm{ as }k\rightarrow \infty,
$$
where $\bar D\in \mathcal{D}$  is a solution of  (\ref{minper}), with datum $u_{meas}$.
\end{proposition}

\begin{proof}
Observe that for any $D_k$
\[
\frac{1}{2} \int_{\Sigma}(u(D_k) - u_k)^2d\sigma+\alpha \textrm{Per}(D_k)\leq \frac{1}{2} \int_{\Sigma}(u(D) - u_k)^2d\sigma+\alpha \textrm{Per}(D),\,\,\, \forall D\in\mathcal{D}
\]
Hence, $\textrm{Per}(D_k)\leq K$ and hence, possibly up to subsequences, we have that
 $$
d_H(D_k,\bar D)\rightarrow 0,\,\, k\rightarrow \infty
$$
for some $\bar D\in\mathcal{D}$ and
$$
\textrm{Per}(\bar D)\leq \liminf_{k\rightarrow\infty}\textrm{Per}(D_k).
$$
Furthermore, by Corollary 3.3 
$$
u(D_k)\rightarrow u(\bar D),\,\,k\rightarrow \infty \textrm{ in } L^2(\Sigma),
$$
implying
\[
\begin{aligned}
J(\bar D)&\leq \liminf_{k\rightarrow\infty}\frac{1}{2} \int_{\Sigma}(u(D_k) - u_k)^2d\sigma+\alpha \textrm{Per}(D_k)\\
&\leq  \lim_{k\rightarrow\infty}\frac{1}{2} \int_{\Sigma}(u(D) - u_k)^2d\sigma+\alpha \textrm{Per}(D)=\frac{1}{2} \int_{\Sigma}(u(D) - u_{meas})^2d\sigma+\alpha \textrm{Per}(D),\,\,\,\forall D\in \mathcal{D}.
\end{aligned}
\]

and the claim follows.
\end{proof}
Now we prove that the solution to the minimization problem (\ref{minper}) converges as $\alpha\rightarrow 0$ to the unique solution of the inverse problem  defined at the beginning of this section.
\begin{proposition}
Assume a solution $\tilde D\in \mathcal{D}$  to the inverse problem corresponding to datum $u_{meas}$ exists. For any $\eta>0$ let  $(\alpha(\eta))_{\eta>0}$ be  such that  $\alpha(\eta)=o(1)$ and $\frac{\eta^2}{\alpha(\eta)} $ is bounded as $ \eta\rightarrow 0$.\\
Furthermore, let $D_{\eta}$ be a solution to the minimization problem (\ref{minper}) with $\alpha=\alpha(\eta)$ and datum $u_{\eta}\in L^2(\Sigma)$ satisfying $\|u_{meas}-u_{\eta}\|_{L^2(\Sigma)}\leq \eta$. Then
$$
D_{\eta}\rightarrow \tilde D
$$
in the Hausdorff metric as $\eta\rightarrow 0$.

\end{proposition}

\begin{proof}
Consider the solution $\tilde D$ to the inverse problem corresponding to the datum $u_{meas}$. By definition of $D_{\eta}$,
\begin{equation}\label{11}
\begin{aligned}
\frac{1}{2} \int_{\Sigma}(u(D_{\eta}) - u_{\eta})^2d\sigma+\alpha\textrm{Per}(D_{\eta})&\leq \frac{1}{2} \int_{\Sigma}(u(\tilde D) - u_{\eta})^2d\sigma+\alpha \textrm{Per}(\tilde D)\\
&=\frac{1}{2} \int_{\Sigma}(u_{meas} - u_{\eta})^2d\sigma+\alpha \textrm{Per}(\tilde D)\leq \eta^2+\alpha \textrm{Per}(\tilde D)
\end{aligned}
\end{equation}
In particular
\[
 \textrm{Per}(D_{\eta})\leq \frac{\eta^2}{\alpha}+\textrm{Per}(\tilde D)\leq C
\]
Hence, arguing as in Proposition \ref{mincav}, possibly up to subsequences,
$$
d_H(D_{\eta}, D_0)\rightarrow 0, \textrm{as } \eta\rightarrow 0
$$
for some $ D_0\in \mathcal{D}$. Passing to the limit in (\ref{11}) as $\eta\rightarrow 0$ we derive
$$
 \int_{\Sigma}(u(D_{\eta}) - u_{\eta})^2d\sigma\rightarrow 0,
 $$
 hence, also
$$
 \frac{1}{2}\int_{\Sigma}(u(D_{\eta}) - u_{meas})^2d\sigma\leq\int_{\Sigma}(u(D_{\eta}) - u_{\eta})^2d\sigma+\int_{\Sigma}(u_{meas}- u_{\eta})^2d\sigma\rightarrow 0.
 $$
 By Corollary \ref{limitdom}, from last relation we have
 $$
 u(D_0)=u_{meas}\,\,\text{ on } \Sigma
 $$
 and by the uniqueness of the inverse problem proved in \cite{BCP} this implies $D_0=\tilde D$ which concludes the proof.
\end{proof}

\bigskip

Before concluding the section, we observe that the minimization problem \eqref{minper} can be equivalently formulated in terms of the indicator function $v$ of $D$ as follows
\begin{equation}\label{minv}
\min_{v\in X_{0,1}} J(v):\, J(v) = \frac{1}{2} \int_{\Sigma}(u(v) - u_{meas})^2d\sigma+\alpha \textrm{TV}(v),
\end{equation}
 where
\[
 X_{0,1}= \{ v \in BV(\Omega): v(x)\equiv\chi_{\omd} \text{ a.e. in $\Omega$ }, D\in\mathcal{D} \},
\]
$u(v):=u(D)$ where $v = \chi_D$ and $u(D)$ is the solution of the boundary value problem \ref{probcav}. 
As a consequence of the proof of Proposition \ref{mincav}, the functional $J$ is lower semicontinuous with respect to the $L^1$ topology.

\subsection{Filling the cavity with a fictitious material}
\label{subsec:delta}

In order to address the minimization problem \eqref{minv} numerically we will follow the approach proposed  in \cite{BC} for a topological optimization problem in linear elasticity, which consists in first filling the cavity $D$ with a fictitious material of very small conductivity $\delta>0$ 
and considering, for $v\in X_{0,1}$, the transmission boundary value problem
\begin{equation}
	\left\{
	\begin{aligned}
		-\textrm{div}(a_{\delta}(v) \nabla u) + vu^3 &= f \qquad \text{in } \Omega	\\
		 \displaystyle{\frac{\partial u}{\partial\mathbf{n}}}=0,\qquad \text{on } \partial \Omega,
	\end{aligned}
	\right.
	\label{eq:probclassic}
\end{equation}
where $a_{\delta}(v)=\delta+(1-\delta)v$.

Note that, from Proposition 2.1 in \cite{BRV}, problem (\ref{eq:probclassic}) has a unique solution $u_\delta=u_{\delta}(v)\in H^1(\Omega)$ for any fixed $\delta>0$. Also, under Assumption \ref{as:4}  on $f$ one can extend the truncation argument introduced in \cite{BCP} to derive a similar estimate as (\ref{boundsu}) in Section 3 obtaining 
\begin{equation}\label{estimateudelta}
\|u_{\delta}\|_{L^{\infty}(\Omega)}\leq C
\end{equation}
and using the variational formulation for $u_{\delta}$ and estimate (\ref{estimateudelta}) 
\begin{equation}\label{estimategradudelta}
\int_{\Omega}a_{\delta}(v_{\delta})|\nabla u_{\delta}|^2\leq C
\end{equation}
where $C$ does not depend on $\delta$.

It is now natural to replace the minimization problem (\ref{minv}) with the following one:
\begin{equation}\label{mindelta}
\min_{v\in X_{0,1}} J_{\delta}(v):\, J_{\delta}(v) = \frac{1}{2} \int_{\Sigma}(u_{\delta}(v) - u_{meas})^2d\sigma+\alpha \textrm{TV}(v).
\end{equation}

In the sequel we will use the following continuity result for solutions to Problem (\ref{eq:probclassic}) with respect to $v\in X_{0,1}$ in the $L^1$  topology.
More precisely:
\begin{proposition}\label{cont}
Let $f$ satisfy Assumption \ref{as:4}  in Section 2. Then if $\{v_n\}$ is a sequence in \\$\tilde{X}=\{v\in L^1(\Omega;[0,1]): v=1 \text{ a.e. in }\Omega_{d_0}\}$ such that $v_n\rightarrow \overline{v}\in\tilde{X}$ in $L^1(\Omega)$ it follows that
\[
\int_{\Sigma}(u_{\delta}^n(v_n) - u_{meas})^2d\sigma\rightarrow\int_{\Sigma}(u_{\delta}(\overline{v}) - u_{meas})^2d\sigma,\,\, \text{ as }n\rightarrow \infty
\]
where $u_{\delta}^n(v_n)$ denotes the solution to \eqref{eq:probclassic} corresponding to $v=v_n$ and $u_{\delta}(\overline{v})$ denotes the solution to \eqref{eq:probclassic} corresponding to $v=\overline{v}$.
\end{proposition}
\begin{proof}
Let $w_n=u_n- \overline{u}$ where, to simplify the notation, we have set  $ \overline{u}:=u_{\delta}(\overline{v})$ and $u_n:=u_{\delta}^n(v_n)$.  Then an easy computation shows that $w_n$ is solution to the problem
\begin{equation}
	\left\{
	\begin{aligned}
		-\textrm{div}(a_{\delta}(v_n) \nabla w_n) + v_nq_nw_n &= \textrm{div}((a_{\delta}(v_n)-a_{\delta}(\overline{v})) \nabla \overline{u})-(v_n-\overline{v})\overline{u}^3 \qquad \text{in } \Omega	\\
		 \displaystyle{\frac{\partial w_n}{\partial\mathbf{n}}}=0,\qquad \text{on } \partial \Omega,
	\end{aligned}
	\right.
	\label{eq:probclassic2}
\end{equation}
where $q_n=u^2_n+u_n\overline{u}+\overline{u}^2$. Multiplying the equation by $w_n$ and integrating by parts over $\Omega$ we obtain

\begin{equation}
\int_{\Omega}a_{\delta}(v_n) |\nabla w_n|^2+\int_{\Omega}v_nq_nw_n^2
=\int_{\Omega}\left(a_{\delta}(v_n)-a_{\delta}(\overline{v})\right) \nabla \overline{u}\cdot\nabla w_n-\int_{\Omega}(v_n-\overline{v})\overline{u}^3w_n.
\label{eq:diff}
\end{equation}
Set
\[
I:=\int_{\Omega}a_{\delta}(v_n) |\nabla w_n|^2+\int_{\Omega}v_nq_nw_n^2.
\]
Then
\begin{equation}\label{I}
I\geq \delta \int_{\Omega} |\nabla w_n|^2+\int_{\Omega_{d_0}}q_nw_n^2\geq \delta\left(\int_{\Omega} |\nabla w_n|^2+\int_{\Omega_{d_0}}q_nw_n^2\right)\geq C \left(\|\nabla w_n\|^2_{L^2(\Omega)}+\|w_n\|^2_{L^2(\Omega_{d_0})}\right).
\end{equation}
Last inequality  on the right hand side follows by an application of Poincar\'e inequality (see Theorem A1 in \cite{BCMP}). In fact, setting  $g_n=\frac{1}{\int_{\Omega_{d_0}}\,q_n}q_n\chi_{\Omega_{d_0}}$ and $\overline{w}_n=\int_{\Omega_{d_0}}g_nw_n$ we have
\[
\|w_n\|^2_{L^2(\Omega_{d_0})}\leq 2(S^2\|\nabla w_n\|^2_{L^2(\Omega)}+|\Omega_{d_0}|\overline{w}^2_n).
\]
Observe now that
\[
\overline{w}^2_n\leq\frac{1}{\int_{\Omega_{d_0}}q_n}\int_{\Omega_{d_0}}q_nw^2_n
\]
and that
$$
q_n\geq \frac{3}{4}\overline{u}^2.
$$
Hence,
$$
\int_{\Omega_{d_0}}q_n \geq  \frac{3}{4}\int_{\Omega_{d_0}}\overline{u}^2=m_0>0,
$$
as,  if $m_0=0$ this would imply $\overline{u}=0$ a.e. in $\Omega_{d_0}$. Then, the equation for $\overline{u}$ and $v=1$ a.e.  in $\Omega_{d_0}$  would imply $f=0$ a.e.  in $\Omega_{d_0}$ which is a contradiction. Therefore
\[
\overline{w}^2_n\leq \frac{1}{m_0}\int_{\Omega_{d_0}}q_nw^2_n
\]
and
\[
\|w_n\|^2_{L^2(\Omega_{d_0})}\leq
C\left(\|\nabla w_n\|^2_{L^2(\Omega)}+\int_{\Omega_{d_0}}q_nw^2_n\right)
\]
which trivially implies the last inequality in (\ref{I}).
Moreover, again by Poincar\'e inequality

\begin{align}
\|w_n\|^2_{H^1(\Omega)}\leq C\left(\|\nabla w_n\|^2_{L^2(\Omega)}+\|w_n\|^2_{L^2(\Omega_{d_0})}\right)\leq
\int_{\Omega}|(a_{\delta}(v_n)-a_{\delta}(\overline{v})| |\nabla \overline{u}||\nabla w_n|+\int_{\Omega}|(v_n-\overline{v})\,\overline{u}^3\,w_n|\\
\leq C\left(\left(\int_{\Omega}|v_n-\overline{v}|^2 |\nabla \overline{u}|^2\right)^{1/2}+
\left(\int_{\Omega}|v_n-\overline{v}|^2 | \overline{u}|^6\right)^{1/2}\right)\|w_n\|_{H^1(\Omega)}
\end{align}

which gives
\[
\|w_n\|_{H^1(\Omega)}\leq C\left(\left(\int_{\Omega}|v_n-\overline{v}|^2 |\nabla \overline{u}|^2\right)^{1/2}+\left(\int_{\Omega}|v_n-\overline{v}|^2 | \overline{u}|^6\right)^{1/2}\right).
\]
Finally, since  $v_n\rightarrow \overline{v}$ a.e. in $\Omega$, applying the  dominated convergence theorem  we get that as $n\rightarrow \infty$
\[
\|w_n\|_{H^1(\Omega)}\rightarrow 0
\]
which implies that
\[
\|w_n\|_{L^2(\Sigma)}\rightarrow 0
\]
and the claim follows.
\end{proof}

\subsection{Phase-field relaxation}
\label{subsec:eps}

In practice it is still difficult to address the minimization problem \eqref{mindelta} numerically because of the non-differentiability of the cost functional and the non-convexity of the space $X_{0,1}$.  So, we will consider a further regularization in which the total variation is approximated by a Ginzburg-Landau type of energy and the space $X_{0,1}$ is substituted with a more regular convex space of phase field variables. This kind of approach has been used extensively in shape and topology optimization and also in the context of inverse problems (see for example \cite{BRV}, \cite{BZ}, \cite{BC}, \cite{R}, \cite{LY}, \cite{DES}).

Hence, we introduce a phase field relaxation of the total variation which will allow to operate with more regular functions with values in $[0,1]$ and formulate a relaxed version of Problem (\ref{mindelta}) similarly as in \cite{BRV}.\\ 
To this purpose, we define the following subset of $H^1(\Omega)$ 
\[ \mathcal{K} = \{ v \in H^1(\Omega;[0,1]):  v=1 \text{ a.e in }\Omega_{d_0}\}.\]\\
For every $\varepsilon >0$, we will consider the optimization problem:
\begin{equation}\label{minrel}
\min_{v \in \mathcal{K}} J_{\delta,{\varepsilon}}(v); \quad J_{\delta,{\varepsilon}}(v) =\frac{1}{2} \int_{\Sigma}(u_{\delta}(v) - u_{meas})^2d\sigma+ \alpha \int_{\Omega}\left( \gamma\varepsilon|\nabla v|^2 + \frac{\gamma}{\varepsilon} v^2(1-v)^2 \right), %
\end{equation}
where $\gamma$ is a suitable normalization constant.
We have the following
\begin{proposition}
\label{mindelteps}
 For every fixed $\delta>0$ and $ \varepsilon > 0$, the minimization problem (\ref{minrel}) has a solution $v_{\delta,\varepsilon}, \in \mathcal{K}$. Furthermore, if $u^n_{meas}\rightarrow u_{meas}$ in $L^2(\Sigma)$ and $v^n_{\delta,\epsilon}$ denotes the solution of Problem (\ref{minrel}) with datum $u_{meas}^n$, then possibly up to a subsequence, we have that $v^n_{\delta,\epsilon}\rightarrow v_{\delta,\varepsilon}$ in $H^1(\Omega)$ where $v_{\delta,\varepsilon}\in \mathcal{K}$
is solution of Problem (\ref{minrel}) with datum $u_{meas}$.
\end{proposition}
We omit the proof of Proposition \ref{mindelteps} since it can be found in \cite{BRV}, see Proposition 2.6 and Proposition 2.7 therein.

At this stage it is natural to address the following two problems: the possible    $\Gamma$-convergence of $J_{\delta,{\varepsilon}}$ to  $J_{\delta}$   (defined in \eqref{mindelta}) as $\varepsilon \rightarrow 0$ and the $\Gamma$-convergence of $J_{\delta}$ to $J$ (defined in (\ref{minv})) as $\delta \rightarrow 0$. This would imply, thanks to the fundamental theorem of $\Gamma$-convergence that  minima of $J$ could be approximated by minima of $J_{\delta,{\varepsilon}}$ for $\varepsilon$ and $\delta$ sufficiently small. In order to prove the convergence in $\varepsilon$ we need to adapt the proof of Modica Mortola to our case but, 
compared to the analysis in \cite{DES} and \cite{BRV}, here the functional $J_{\delta}$ is defined on a  subset, $X_{0,1}\subset BV(\Omega)$, of characteristic functions of more regular domains. For this reason, we are forced to restrict the relaxed functional $J_{\delta,\varepsilon}$  to some suitable subset of $\mathcal{K}$. 

For $\eta\in(0,1)$ let us define  
\[
\mathcal{K}_{\eta}=\{v\in\mathcal{K}: \{v\geq \eta \}={\Omega}_D \text{ a. e. } \text{for some }D\in\mathcal{D}\}
\]
Note that though $\Omega_D$ is an open set the definition makes sense since $\partial D$ has zero Lebesgue measure.
Though this set is not convex let us notice that it is a weakly closed subset of $\mathcal{K}$ with respect to the $H^1(\Omega)$ topology guaranteeing the existence of minima of the functional $J_{\delta,{\varepsilon}}$ in $K_{\eta}$. In fact, we can prove  
\begin{lemma}\label{wclosed}
Let $\{v_k\}\in \mathcal{K}_{\eta}$ be a sequence converging weakly in $H^1(\Omega)$ to an element $v$. Then $v\in\mathcal{K}_{\eta}$.
\end{lemma}

\begin{proof}
 Let us define ${\Omega}_{D_k}:=\{v_k\geq \eta \}$. By Proposition \ref{compactD}, possibly up to a subsequence, we have that 
\[D_k\rightarrow D_0\in \mathcal{D}\,\,\,and\,\,\,\,\partial D_k\rightarrow \partial D_0
\]
 in the Hausdorff topology. Then, $D_k$ also converges to $D_0$ in measure which implies that 
\[\chi_{{\Omega}_{D_k} }\rightarrow \chi_{{\Omega}_{D_0}}
\]
in $L^1(\Omega)$ and almost everywhere in $\Omega$. Let us now show that
\[
\{v\geq \eta\}={\Omega}_{D_0}, \textrm{ a.e.}
\]
In fact, since
\[
\eta\chi_{\Omega_{D_k}}\leq v_k\leq \eta+(1-\eta)\chi_{\Omega_{D_k}}
\]
and noting that $v_k$, possibly up to a subsequence, converges a.e. to $v$, passing to the limit pointwise we have
\[
\eta\chi_{\Omega_{D_0}}\leq v\leq \eta+(1-\eta)\chi_{\Omega_{D_0}}\,\,\text {a.e.}
\]
From the left-hand side inequality it follows that ${\Omega}_{D_0}\subseteq \{v\geq \eta \}$ while from the right-hand side inequality we get that
$ \{v\geq \eta\}\subseteq {\Omega}_{D_0}$ concluding the proof. 
\end{proof}
As an immediate consequence we get
\begin{corollary}
For every fixed $\delta>0$ and $ \varepsilon > 0$, the minimization problem $$\min_{v \in \mathcal{K}_{\eta}} J_{\delta,{\varepsilon}}(v)$$ where $J_{\delta,{\varepsilon}}$ is defined in \eqref{minrel} has a solution $v_{\delta,\varepsilon}$. Furthermore, if $u^n_{meas}\rightarrow u_{meas}$ in $L^2(\Sigma)$ and $v^n_{\delta,\epsilon}$ denotes a solution with datum $u_{meas}^n$, then possibly up to a subsequence, we have that $v^n_{\delta,\epsilon}\rightarrow v_{\delta,\varepsilon}$ in $H^1(\Omega)$ where $v_{\delta,\varepsilon}\in \mathcal{K}_{\eta}$
is a solution with datum $u_{meas}$.
\end{corollary}

\subsection{Analysis of the $\Gamma-$limits}
\label{subsec:gamma}

We now investigate the asymptotic properties of the introduced functionals: in particular, we first concentrate on the limit of $J_{\delta,{\varepsilon}}$ as $\varepsilon \rightarrow 0$, in the sense of $\Gamma-$convergence. The proof of the next theorem will clarify our choice of the subset $\mathcal{K}_{\eta}$. 
For $v\in L^1(\Omega)$,  consider the following extensions of the cost functionals
\begin{equation}
 \tilde{J}_{\delta}(v) =
\left\{
	\begin{aligned}
	J_{\delta}(v) & \quad \textit{if $v \in X_{0,1}$}\\
	\infty & \quad \textit{otherwise in } L^1(\Omega)
	\end{aligned}
\right.
\label{minregbis1}
\end{equation}

and of (\ref{minrel})
\begin{equation}
 \tilde{J}_{\delta,\varepsilon}(v) =
\left\{
	\begin{aligned}
	J_{\delta,\varepsilon}(v) & \quad \textit{if $v \in \mathcal{K}_{\eta}$}\\
	\infty & \quad \textit{otherwise in }  L^1(\Omega)
	\end{aligned}
\right.
\label{minrel2}
\end{equation}
Then
\begin{theorem} \label{thm:Gconv_eps}
Consider a sequence $\{\varepsilon_k\}$ s.t. $\varepsilon_k \rightarrow 0$ as $k\rightarrow +\infty$. Then, the functionals $\tilde{J}_{\delta,\varepsilon_k}$ converge to $\tilde{J}_{\delta}$ in $L^1(\Omega)$ in the sense of the $\Gamma-$convergence.
\end{theorem}
\begin{proof}
Write
$$
\tilde{J}_{\delta,\varepsilon}(v)=F_{\delta}(v)+G_{\varepsilon}(v)
$$
where 
\[F_{\delta}(v):= \frac{1}{2} \int_{\Sigma}(u_{\delta}(v) - u_{meas})^2d\sigma\] for any $v\in L^1(\Omega)$ and 
\[G_{\varepsilon}(v):= \int_{\Omega}\left( \gamma\varepsilon|\nabla v|^2 + \frac{\gamma}{\varepsilon}v^2(1-v)^2\right)\] 
for $v\in\mathcal{K}_{\eta}$ and $\infty$ otherwise in $L^1(\Omega)$.\\ 
Then from the continuity of $F_{\delta}(v)$  in $L^1(\Omega)$ derived in Proposition \ref{cont} it is enough to show  $\Gamma$-convergence of $G_{\varepsilon}$ as $\varepsilon \rightarrow 0$ to $\operatorname{TV}(v)$. Then by Remark 1.7 in \cite{dalM} it follows that
$
\tilde{J}_{\delta,\varepsilon}$ $\Gamma$-converges to $\tilde{J}_{\delta}$ as $\varepsilon \rightarrow 0$. \\
Let us prove $\Gamma$-convergence of $G_{\varepsilon}$ as $\varepsilon \rightarrow 0$.\\ 
\it{(i)} We first prove the $\liminf$ property i.e. for every sequence $\varepsilon_k\rightarrow 0$ and for every sequence $\{v_k\} \subset L^1(\Omega)$ s.t. $v_k \xrightarrow{L^1} v $, 
\[\operatorname{TV}(v) \leq \liminf_k {G}_{\varepsilon_k}(v_{k})\]. 

Consider a sequence $v_{k}$ converging in $L^1(\Omega)$ to a function $v\in L^1(\Omega)$ as $\varepsilon_k\rightarrow 0$ for $k\rightarrow \infty$.  
We can assume that $\{v_k\}\in \mathcal{K}_{\eta}$ since otherwise the claim would follow trivially. Moreover, by [MM] we know that  $v=\chi_{\Omega_D}$ for some $D\subset\Omega$ with finite perimeter
and $\operatorname{TV}(v)\le \liminf_k {G}_{\varepsilon_k}(v_{k})$.
\\
Finally, by reasoning as in the proof of Lemma \ref{wclosed} we obtain the following relation a.e.
\[
\eta\chi_{\Omega_{D_0}}\leq \chi_{\Omega_D}\leq \eta+(1-\eta)\chi_{\Omega_{D_0}},
\]
for some $D_0\in \mathcal D$, from which it follows that a.e.
\[
\chi_{\Omega_D}=\chi_{\Omega_{D_0}}\in X_{0,1}.
\]
\textit{(ii)} Let us now prove the $\limsup$ property: for any $v\in L^1(\Omega)$ there exists a sequence $\{v_k\}$ converging to $v$ in $L^1(\Omega)$ as $k\rightarrow +\infty$ such that 
\[\limsup_k G_{\varepsilon_k}(v_{k})\leq \operatorname{TV}(v).\] 
Let us first observe that the above inequality can be checked on a suitable dense set of $X_{0,1}$, see for example Remark 1.29 of \cite{B}. Therefore, we may only consider the set
\[\mathcal{L}=\{ \chi_{\Omega_D}, D\subset\Omega,\,\,D\in \mathcal{D},\,\, \partial D\in C^{\infty}\}.\]
Actually, by Theorem 1.12 in \cite{Ver} (see in particular properties (ii) and (iii)) it follows that for any set $D\in\mathcal{D}$ there exists a sequence of smooth domains $D_k\in\mathcal{D}$ (i.e. $\forall k$, $\partial D_k$ is $\mathcal{C}^{\infty}$ and satisfies assumption $3$ with constants $r_0$, $L_0$) such that 

\[
{\partial{D_k}}\rightarrow {\partial D}\,\,\textrm{ in\,\,the\,\,Hausdorff\,\,metric }\textrm{ as }k \rightarrow +\infty
\]

In particular, this implies
\begin{enumerate}
\item[(1)] 
\[
\chi_{\Omega_{D_k}}\rightarrow \chi_{\Omega_D}\,\,\textrm{ in }L^1(\Omega)\textrm{ as }k \rightarrow +\infty
\]
\item[(2)] 
\[\mathcal{H}^1(\partial D_k)\rightarrow \mathcal{H}^1(\partial D),\,\,\textrm{as }k \rightarrow +\infty\]
\end{enumerate}

In order to check the $\limsup$  property on $ \mathcal{L}$, we follow the standard approach by constructing a suitable recovery sequence. Hence, let us consider the Cauchy problem
\[
\begin{cases}
g'=g(1-g)\,\textrm{in }\mathbb{R}\\
g(0)=\eta\in (0,1)
\end{cases}
\]
Note that the solution is globally defined, $0<g<1$ and $g$  has limits $0$ and $1$ for $t\to -\infty$ and $t\to +\infty$ respectively. Now, for any $\beta>0$ we take $M_{\beta}>0$ such that $g(M_{\beta})\le 1-\beta$ and $g(M_{\beta})\le 0$ and consider the function

\[
g_{\beta}(t):=\begin{cases}
1,\,\,t\in (M_{\beta}+1,\infty)\\
(1-g(M_{\beta}))(t-M_{\beta})+g(M_{\beta}),\,\,t\in [M_{\beta},M_{\beta}+1]\\
g(t),\,\,t\in [-M_{\beta},M_{\beta}]\\
g(-M_{\beta})(t+M_{\beta}+1),\,\,t\in [-M_{\beta}-1,-M_{\beta}]\\
0,\,\,t\in (-\infty,-M_{\beta}-1)
\end{cases}
\]
Let us now fix $v\in \mathcal{L}$ i.e. $v:=\chi_{\Omega_D}$ with $D$ smooth and define the signed distance from $\partial D$
\[
\rho(x):=
\begin{cases}
-\inf_{y\in\partial D}d(x,y),\,\, x\in D\\
\inf_{y\in\partial D}d(x,y),\,\, x\in \Omega_D.
\end{cases}
\]
Then we can define 
\[
v_{\varepsilon,\beta}(x):=
\begin{cases}
1,\,\, \textrm{in }\{x\in\Omega:\rho(x)\geq (M_{\beta}+1)\varepsilon\}\\
g_{\beta}\left(\frac{\rho(x)}{\varepsilon}\right),\,\,\textrm{in }\{x\in\Omega:|\rho(x)|\leq (M_{\beta}+1)\varepsilon\}\\
0,\,\,\textrm{in }\{x\in\Omega:\rho(x)\leq -(M_{\beta}+1)\varepsilon\}.
\end{cases}
\]
A crucial observation is that for every positive and small enough $\varepsilon, \beta$, the function $v_{\varepsilon,\beta}\in \mathcal{K}_{\eta}$.

In fact, by definition we have $v_{\varepsilon,\beta}\in H^1(\Omega;[0,1])$ and
$v_{\varepsilon,\beta}(x)=g_{\beta}(0)=g(0)=\eta$ for $x\in \partial D$. Moreover, since $g_{\beta}$
is a strictly increasing function of the signed distance from $\partial D$ (and by recalling that $|\partial D|=0$) we readily get
$\{v_{\varepsilon,\beta}\geq \eta \text{ a.e.}\}=\Omega_{D}$, with $D\in \mathcal{D}$, so that the claim follows. 

Now, by standard arguments, see for example \cite{M} and \cite{MM}, we can find a sequence $\{v_{\varepsilon_k,\beta_k}\}_{k=1}^{\infty}$ (with $\beta_k\rightarrow 0,\,\, \varepsilon_k\rightarrow 0$) converging in $L^1(\Omega)$ to $v$ and satisfying the $\limsup$ property.  

\end{proof}
As a consequence, from the equicoerciveness of the functionals $G_{\varepsilon}$  and by the $\Gamma$-convergence, see for example Theorem 7.4 in \cite{dalM}, we derive the following convergence result for the solutions of Problem (\ref{minrel2}).
\begin{corollary}
\label{convmin}
Assume $\delta>0$, $\varepsilon>0$, and let $v_{\delta,\varepsilon}$ be a minimum of the functional (\ref{minrel2}).  Then there exists a sequence $\varepsilon_k\rightarrow 0$ as $k\rightarrow +\infty$ and a function $v_{\delta}\in X_{0,1}$ such that $v_{\delta, \varepsilon_k}\rightarrow v_{\delta}$ in $L^1(\Omega)$ and   $v_{\delta}$ is a minimizer of (\ref {minregbis1}).
\end{corollary}
\vskip 3truemm

\begin{remark} \label{rem:potential}
It can be proved that all the previous results, starting from Proposition \ref{mindelteps}, also hold by replacing in the relaxed functional $G_{\epsilon}(v)$ the potential $v^2(1-v)^2$ with any positive function vanishing only for $v=0$ and $v=1$.
\end{remark}

We are now left with proving the $\Gamma$-convergence of $J_{\delta}$ to $J$  as defined in \eqref{minv}.
 To this aim we need to prove some preliminary results. The first concerning properties of the set $X_{0,1}$ and the second regarding continuity properties of solutions to \ref{eq:probclassic} as $\delta\rightarrow 0$ that we derive adapting the proof of Theorem 4.2 in \cite{R}.
 
 \begin{lemma}\label{closeX}
 The set $X_{0,1}$ is closed in the $L^1(\Omega)$ topology.
 \end{lemma}
 \begin{proof}
 Consider a sequence $\{v_n\}\in X_{0,1}$ and assume that $v_n\rightarrow v$ in $L^1(\Omega)$. Then, possibly up to subsequences, $v_n=\chi_{\Omega_{D_n}}\rightarrow v$ pointwise a.e. in $\Omega$ where $D_n\in\mathcal{D}$. Hence, it follows that $v=\chi_{\Omega_D}$ for some measurable set $D$. Also, by Proposition \ref{compactD} it follows that, possibly up to subsequences, $D_n$ converges in the Hausdorff topology to $D_0\in \mathcal{D}$. Obviously, this also implies that $\chi_{D_n}\rightarrow \chi_{D_0}$ in $L^1(\Omega)$. Hence,  up to a set of measure zero $D=D_0$ implying that $v\in X_{0,1}$.
 
 \end{proof}
 
\begin{proposition}\label{convudeltan}
Under Assumptions \ref{as:1} - \ref{as:4} in  Section 2,  let $\{v_{\delta_n}\}_{n\geq 1}$ be a sequence of elements in $X_{0,1}$ converging in $L^1(\Omega)$ as $\delta_n\rightarrow 0$ to $v$.
Then $v={\chi}_{\Omega_{D}}$ a.e. with $D\in\mathcal{D}$ and the traces on $\Sigma$ of the corresponding solutions to (\ref{eq:probclassic}), $u_{\delta_n}(v_{\delta_n})|_\Sigma$, converge strongly in $L^2(\Sigma)$ to $\tilde{u}|_\Sigma$, where $\tilde{u}$ is the solution to problem \eqref{probcav} with cavity $D$.
\end{proposition}
Since the proof of Proposition \ref{convudeltan} is long and rather technical and instrumental to get the $\Gamma-$ convergence result, we have preferred to put it in the Appendix. 
 
 The previous proposition can be used to prove the $\Gamma$-convergence, as $\delta_n \rightarrow 0$, of the functionals $\tilde{J}_{\delta}$ defined in \eqref{minregbis1} to the limit functional
  \begin{equation}
 \tilde{J}(v) =
\left\{
	\begin{aligned}
	J(v) & \quad \textit{if $v \in X_{0,1}$}\\
	\infty & \quad \textit{otherwise in } L^1(\Omega),
	\end{aligned}
\right. \qquad J(v) = \frac{1}{2} \int_{\Sigma}(u(v) - u_{meas})^2d\sigma+\alpha \textrm{TV}(v), 
\label{minregbis}
\end{equation}
where $u(v)$ is the solution of \eqref{probcav} with cavity $D$ such that $v = \chi_{\Omega_D}$.

 \begin{theorem}
 \label{thm:Gconv_del}
Consider a sequence $\{\delta_n
\}$ s.t. $\delta_n \rightarrow 0$. Then, the functionals $\tilde{J}_{\delta_n}$ converge to $\tilde{J}$ in $L^1(\Omega)$ in the sense of the $\Gamma$-convergence.
\end{theorem}
\begin{proof}
\it{(i)} We first prove the $\liminf$ property i.e. for every sequence $\delta_n\rightarrow 0$ and for every sequence $\{ v_{\delta_n} \} \subset L^1(\Omega)$ s.t. $v_{\delta_n} \xrightarrow{L^1} v $, $\tilde{J}(v) \leq \liminf_n \tilde{J}_{\delta_n}(v_{\delta_n})$. Consider a sequence $v_{\delta_n}$ converging in $L^1(\Omega)$ to a function $v\in L^1(\Omega)$ as $\delta_n\rightarrow 0$ for $n\rightarrow \infty$.  Then we can assume that
\begin{equation}\label{uniformbound}
J_{\delta_n}(v_{\delta_n})\leq C.
\end{equation}
In fact, if  $\liminf_n \tilde{J}_{\delta_n}(v_{\delta_n})=+\infty$ then the $\liminf$ property trivially follows. Hence, possibly up to a subsequence, $ \liminf_n \tilde{J}_{\delta_n}(v_{\delta_n})=\lim_n \tilde{J}_{\delta_n}(v_{\delta_n})<+\infty$ which implies \ref{uniformbound}. Then $v_{\delta_n}\in X_{0,1}$ and
$$
\operatorname{TV}(v_{\delta_n})\leq C
$$
and by the lower semicontinuity of the total variation with respect to the $L^1$ convergence we have that
\begin{equation}\label{per1}
\operatorname{TV}(v)\leq \liminf_{n\rightarrow +\infty}\operatorname{TV}(v_{\delta_n})\leq C
\end{equation}
Also,  possibly up to subsequences, since $v_{\delta_n}=\chi_{\Omega_{D_n}}\rightarrow v$ a.e. in $\Omega$, it follows that $v=\chi_{\Omega_D}$ and $v=1\textrm{ in }\Omega_{d_0}$.  Furthermore, by Lemma \ref{closeX} we have that $v\in X_{0,1}$. 
Finally, since $v_{\delta_n} \in X_{0,1}$, 
we can use Proposition \ref{convudeltan} to conclude that

\begin{equation}\label{contfunct1}
\int_{\Sigma}|u_{\delta_n}(v_{\delta_n})-u_{meas}|^2\rightarrow \int_{\Sigma}|u(v)-u_{meas}|^2
\end{equation}
as $n\rightarrow +\infty$.
Hence, using (\ref{per1}) and (\ref{contfunct1}) we have
\begin{equation}\label{liminf}
J(v)\leq \liminf_{n\rightarrow +\infty} J_{\delta_n}(v_{\delta_n}).
\end{equation}
\textit{(ii)} Let us now prove the following property equivalent to the $\limsup$ property: for any $v\in L^1(\Omega)$ there exists a sequence $\{v_{\delta_n}\}$ converging to $v$ in $L^1(\Omega)$ such that $\limsup_{n\rightarrow \infty}\tilde{J}_{\delta_n}(v_{\delta_n})\leq \tilde{J}(v)$.

 Let $v\in L^1(\Omega)$. Then if $\tilde{J}(v)=+\infty$ then the property trivially follows. So, we can assume that $v\in X_{0,1}$. Consider now the following sequence  $\{v_{\delta_n}\}_{n\geq 0}=\{v\}_{n\geq 0}$.
Then
$$
\limsup_{n\rightarrow +\infty}  J_{\delta_n}(v_{\delta_n})=\limsup_{n\rightarrow +\infty}\int_{\Sigma}|u_{\delta_n}(v)-u_{meas}|^2+\alpha \operatorname{TV}(v)
$$
and by Proposition  \ref{convudeltan}, possibly up to subsequences, it follows that
$$
\limsup_{n\rightarrow +\infty}\int_{\Sigma}|u_{\delta_n}(v)-u_{meas}|^2=\lim_{n\rightarrow +\infty}\int_{\Sigma}|u_{\delta_n}(v)-u_{meas}|^2=\int_{\Sigma}|u(v)-u_{meas}|^2
$$
where $u(v)$ is the solution of (\ref{probcav}) corresponding to $v=\chi_{\Omega_D}$ and hence we finally obtain that
$$
\limsup_{n\rightarrow +\infty}  J_{\delta_n}(v_{\delta_n})=\lim_{n\rightarrow +\infty}\int_{\Sigma}|u_{\delta_n}(v)-u_{meas}|^2+\alpha \operatorname{TV}(v)=\int_{\Sigma}|u(v)-u_{meas}|^2+\alpha \operatorname{TV}(v)=J(v)
$$
concluding the proof.
\end{proof}

From \cite{AFP}  it follows that the functionals $J_{\delta}$ are equicoercive in $L^1(\Omega)$ and hence as a consequence of 
the above theorem and the fundamental theorem of $\Gamma$-convergence  (see for example Theorem 7.4 of \cite{dalM}) we have
 \begin{corollary}
For any  $\delta>0$ let  $v_{\delta}$ be a minimizer of (\ref {mindelta}).  Then there exists a sequence $\delta_n\rightarrow 0$ as $n\rightarrow +\infty$  and a function $v=\chi_{\Omega_D}\in X_{0,1}$ such that $v_{\delta_n}\rightarrow v$ in $L^1(\Omega)$ and   $D$ is solution to Problem (\ref {minper}).
\end{corollary}

\section{Reconstruction algorithm}

\label{sec:numerics}
In this section, we describe a numerical algorithm which takes advantage of the relaxation strategy proposed in section \ref{sec:recon} for the reconstruction of cavities. In the first subsection, we analyze an algorithm tackling problem \eqref{minrel} - namely, the minimization of the functional $\Jade$ for fixed values of $\delta,\varepsilon$ - where we replace the potential $v^2(1-v)^2$ in $G_\varepsilon(v)$ by $v(1-v)$. This choice is in line with the assumptions of Remark \ref{rem:potential} and is preferable for numerical reasons, because of the efficiency of the implementation and of the superior performances in the reconstruction.
A similar algorithm, which was proposed in \cite{DES} for a linear equation, has already been studied for the reconstruction of inclusions in the considered nonlinear counterpart in \cite{BRV}: therefore, we summarize here the main convergence results, and outline a more efficient implementation. In the last subsection, instead, we propose an algorithm tackling problem \eqref{minv}, namely, the minimization of $J$ and thus the (stable) reconstruction of cavities.

\subsection{An iterative algorithm for the relaxed problem}
When fixing $\delta, \varepsilon>0$, the problem of minimizing $J_{\delta,\varepsilon}$ over $\mathcal{K}$ is analogous to what discussed for conductivity inclusions in \cite{BRV}, in which $\delta$ is replaced by the parameter $k$ denoting the physical conductivity inside the inclusion. To minimize the relaxed functional $\Jade$, we can take advantage of its differentiability. In particular, it is possible to prove the following result:
\begin{proposition}(see \cite[Proposition 2.10]{BRV})
    Under Assumptions \ref{as:1} - \ref{as:4} in  Section 2, for every fixed $\delta, \varepsilon>0$, the operator $u_\delta \colon \mathcal{K} \rightarrow H^1(\Omega)$ defined by \eqref{eq:probclassic}) is Fréchet-differentiable, and so is $\Jade: \mathcal{K} \rightarrow \R$. Moreover, for every $v \in \mathcal{K}$ and $\vartheta \in H^1(\Omega) \cap L^\infty(\Omega)$,
    \begin{equation}
    \begin{aligned}
        \Jade'(v)[\vartheta] = & \int_{\Omega}(1-\delta) \nabla u_\delta(v) \cdot \nabla p_\delta(v) \vartheta + \int_{\Omega} (1-\delta) u_\delta(v)^3 p_\delta(v) \vartheta \\
        &+  2 \alpha \varepsilon \int_{\Omega}\nabla v \cdot \nabla \vartheta  + \frac{\alpha}{\varepsilon}    \int_{\Omega}(1-2v)\vartheta;
            \end{aligned}
    \label{eq:Frechet_Jade}
    \end{equation}
where $p_\delta\colon \mathcal{K} \rightarrow H^1(\Omega)$ is the solution map of the \textit{adjoint problem}:  
\begin{equation}
    \int_{\Omega} a_\delta(v) \nabla p_\delta(v) \cdot \nabla \psi + \int_{\Omega} a_\delta(v) 3 u_\delta(v)^2 p_\delta(v) \psi = \int_{\partial \Omega}{(u_\delta(v)-u_{meas})\psi} \qquad \forall \psi \in H^1(\Omega).
\label{eq:adjoint}
\end{equation}
\label{prop:Frechet_Jade}
\end{proposition}  
Taking advantage of the differentiability of $\Jade$ and of the convexity of $\mathcal{K}$, we can derive the following necessary optimality condition:
\begin{equation}
\text{if}\quad v^* \in \argmin_{v \in \mathcal{K}} \Jade(v), \qquad \text{then} \quad J_{\delta,\varepsilon}'(v^*)[v-v^*] \geq 0 \quad \forall v \in \mathcal{K}
    \label{eq:OC_Jade}
\end{equation}
Notice that such condition is not sufficient, unless some other properties are verified, such as the convexity of the functional $\Jade$.
We consider the following iterative algorithm, which takes advantage of the Fréchet differentiability of $\Jade$: in particular, the rationale of our strategy is to tackle the minimization of $\Jade$ by means of a sequence of linearized problems at some iterates $v^{(k)}$. The subsequent iterate is computed by minimizing a functional which consists of the first-order expansion of $\Jade$ around $v^{(k)}$ plus a term which penalizes the distance from $v^{(k)}$, due to the local effectiveness of the linearization. A tentative update scheme would read as
\begin{equation} \label{eq:min_move_Jade_expl}
v^{(k+1)} = \argmin_{v \in \mathcal{K}} \left\{ \frac{1}{2\tau_k} \| v - v^{(k)}\|_{L^2(\Omega)}^2 + J_{\delta,\varepsilon}'(v^{(k)})[v-v^{(k)}]  \right\},
\end{equation}
where $\{\tau_k\}$ is a sequence of prescribed step lengths. Since $J_{\delta,\varepsilon}'$ is evaluated in the previous iterate, \eqref{eq:min_move_Jade_expl} corresponds to an explicit scheme, and in the absence of the constraint on $\mathcal{K}$ it would reduce to the explicit Euler discretization of the gradient flow associated with $\Jade$. The explicit treatment of $J_{\delta,\varepsilon}'$ is beneficial for numerical reasons (due to the severe nonlinearity of the differential), but can lead to instabilities, which entails that the choice of the steplenght $\tau_k$ should be very conservative. As already exploited in \cite{DES} and in \cite{BRV}, we can provide a semi-implicit treatment of the derivative by splitting it in a linear part and a nonlinear one, and evaluating only the nonlinear part in the previous iterate. In particular, we define:
\[
\begin{aligned}
\tilde{J}_{\delta,\varepsilon}'(v^{(k)})[v-v^{(k)}] = & \int_{\Omega}(1-\delta) \nabla u_\delta(v^{(k)}) \cdot \nabla p_\delta(v^{(k)}) (v-v^{(k)}) + \int_{\Omega} (1-\delta) u_\delta(v^{(k)})^3 p_\delta(v^{(k)}) (v-v^{(k)}) \\
        &+\frac{\alpha}{\varepsilon} \int_{\Omega}(1-2v^{(k)})(v-v^{(k)}) +  2 \alpha \varepsilon \int_{\Omega}\nabla v \cdot \nabla \vartheta,
\end{aligned}
\]
where only the last term has been treated implicitly. Through this definition, we finally describe our iterative scheme as:
\begin{equation} \label{eq:min_move_Jade}
v^{(k+1)} = \argmin_{v \in \mathcal{K}} \left\{ \frac{1}{2\tau_k} \| v - v^{(k)}\|_{L^2(\Omega)}^2 + \tilde{J}_{\delta,\varepsilon}'(v^{(k)})[v-v^{(k)}]  \right\}
\end{equation}
for prescribed timesteps $\tau_k$.
At each iteration, the algorithm requires to solve an inner minimization problem associated with a quadratic functional on the convex set $\mathcal{K}$, which can be treated by standard tools of convex optimization.
Unfortunately, since the functional $\Jade$ is in general non convex, the convergence of $v^{(k)}$ to a minimizer is not ensured: nevertheless, we aim to prove that, in the limit, the iterates reach a stationary point, namely, an element $v^*$ which satisfies the (necessary) optimality conditions \eqref{eq:OC_Jade}. 
As outlined in \cite{DES}, expression \eqref{eq:min_move_Jade} resembles the discretization of a parabolic obstacle problem, and it is more generally reminiscent of De Giorgi's theory of minimizing movements for differentiable functionals (see \cite[Chapter 7]{braides2014local}). Finally, analogously to what is done in \cite{DES} and \cite{BRV}, we can prove the convergence to a stationary point only in a fully discretized context. 

\begin{remark}
\label{convrelax}
Notice that we are minimizing the functional $\Jade$ in $\mathcal{K}$ instead of $\mathcal{K}_\eta$. This discrepancy between theory and practice is necessary for our purposes, since the non-convexity of the space $\mathcal{K}_\eta$ would not allow to use standard first-order optimization schemes. Nevertheless, numerical evidence from section \ref{sec:results} will show that the algorithm converges to a point belonging to $\mathcal{K}_\eta$: thus, we can consider the minimization within $\mathcal{K}$ as a convex relaxation of the original problem in $\mathcal{K}_\eta$.
\end{remark}

\subsubsection{Discretization of the forward and adjoint boundary value problems}
In order to numerically solve the boundary value problem \eqref{eq:probclassic}, we consider a finite element formulation, which also entails a numerical approximation of the minimization problem \eqref{minrel} we are tackling.
\par In what follows, we introduce a a shape regular triangulation $\mathcal{T}_h$ of $\Omega$, on which we define $V_h \subset H^1(\Omega)$:
\[
    V_h = \{ w_h \in C(\bar{\Omega}), w_h|_K \in \mathbb{P}_1(K) \text{ } \forall K \in \mathcal{T}_h \}; \qquad \mathcal{K}_h = V_h \cap \mathcal{K},
\]
where $\mathbb{P}_1(K)$ denotes the space of polynomials of order $1$ on a domain $K$. 
A discrete counterpart of is provided by considering its weak formulation in $V_h$, which can be interpreted as a nonlinear system of algebraic equations, and approximately solved by means of a Newton-Rhapson algorithm. Moreover, \cite[Proposition 3.1]{BRV} shows that, if we consider an approximation $v_h \in \mathcal{K}_h$ of the indicator function $v \in \mathcal{K}$, and denote the discrete solution associated with $v_h$ as $u_{\delta,h}(v_h)$, then $u_{\delta,h}(v_h) \rightarrow u_\delta(v)$ as the mesh size $h$ reduces. We can analogously introduce the approximate solution $p_{\delta,h}$ of the adjoint equation \eqref{eq:adjoint}, and the discrete version $J_{\delta,\varepsilon,h}$ of the functional \eqref{minrel}, together with its optimality conditions \eqref{eq:OC_Jade}.
Finally, \cite[Proposition 3.4]{BRV} guarantees that, choosing a starting point $v^{(0)}_h \in \mathcal{K}_h$, there exists a collection of timesteps $\{\tau_k\}$ satisfying $0<\tau_{\min} \leq \tau_k \leq \tau_{\max}$ such that the sequence generated by \eqref{eq:min_move_Jade} (where $\Jade$ is replaced by $J_{\delta,\varepsilon,h}$) converges in $W^{1,\infty}$ up to a subsequence to a point satisfying the discrete optimality conditions. 
\subsubsection{Implementation aspects}
By means of \cite[Proposition 3.4]{BRV} and of the ancillary result \cite[Lemma 3.2]{BRV}, we know that the discretized functional $\Jade$ reduces along the iterates of \eqref{eq:min_move_Jade} for a value of $\tau_k$ which is sufficiently small, within an interval $[\tau_{\min}, \tau_{\max}]$. This suggest the possibility to enhanche the iterative algorithm with an adaptive choice of the steplength $\tau_k$, which allows to enlarge it - and therefore save iterations - or to reduce it to guarantee the decrease of the functional across the iterations: 
\begin{algorithm}{Reconstruction of critical points of $\Jade$} \label{alg:Jade}
\begin{enumerate}
    \item choose an initial guess $v^{(0)}_h \in \mathcal{K}_h$ and a step size $\tau_0$
    \item for $k = 0, \ldots, K_{\max}$
    \begin{itemize}
        \item \textbf{if} $k==0$
        \item[] \begin{itemize}
        \item set $\tilde{v}_h^{(k+1)} = v_h^{(k)}$;
        \end{itemize}
        \item[] \textbf{else} 
        \item[] \begin{itemize}
            \item compute $\tilde{v}_h^{(k+1)}$ from $v_h^{(k)}$, $\tilde{J}'_{\delta,\varepsilon,h}(v_h^{(k)})$ and $\tau_k$ via \eqref{eq:min_move_Jade};
        \end{itemize}
        \item compute $u_{\delta,h}(\tilde{v}_h^{(k+1)})$ and $J_{\delta,\varepsilon,h}(\tilde{v}_h^{(k+1)})$
        \item \textbf{if} $\Jade(\tilde{v}_h^{(k+1)}) > \Jade(v_h^{(k)})$
        \item[]
        \begin{itemize}
            \item reduce the steplength $\tau_k$;
        \end{itemize}
        \item[] \textbf{else} 
        \item[] \begin{itemize}
            \item increase the steplength $\tau_k$;
            \item accept the iteration: $v_h^{(k+1)} = \tilde{v}_h^{(k+1)}$ and $u_{\delta,h}(v_h^{(k+1)}) = u_{\delta,h}(\tilde{v}_h^{(k+1)})$
            \item compute $p_{\delta,h}(v_h^{(k+1)})$ and $\tilde{J}'_{\delta,\varepsilon,h}(v_h^{(k+1)})$
            \item check the stopping criterion on $v_h^{(k+1)}$ and increase $k$;
        \end{itemize}
    \end{itemize}
\end{enumerate}
\end{algorithm}
The algorithm is moreover coupled with an adaptive mesh refinement routine: indeed, the iterates $v_h^{(k)}$ are expected to show some regions of diffuse interface, approximating the jump set of the indicator function of the true cavity. The thickness of such regions scales according to $\varepsilon$: in order to precisely capture the support of the gradient, without excessively increasing the total number of elements in $\mathcal{T}_h$, we locally refine the mesh according to the gradient of $v_h^{(k)}$ every $N_{\operatorname{adapt}}=30$ steps. We nevertheless fix a minimum size $h_{\min}=10^{-3}$.

\begin{remark}
In Algorithm \ref{alg:Jade}, the (tentative) update $\tilde{v}_h^{(k+1)}$ is computed by solving the semi-implicit scheme \eqref{eq:min_move_Jade}. As in \cite{DES} and in \cite{BRV}, this is done by means of the Primal-Dual Active Set algorithm, which requires a small number of (sub)iterations. According to the interpretation proposed in \cite{hintermuller2002primal}, we can consider PDAS as a generalized Newton's algorithm for the solution of \eqref{eq:min_move_Jade}, where the constraint is included in the form of a Lagrange multiplier. Alternatively, we observe that the explicit scheme \eqref{eq:min_move_Jade_expl} also admits the following alternative formulation:
\begin{equation}
v^{(k+1)} =\proj_\mathcal{K}\big(v^{(k)}- \tau_k \nabla \Jade(v^{(k)})\big)
    \label{eq:projGrad}
\end{equation}
where $\proj_\mathcal{K}$ is the orthogonal projection on the compact set $\mathcal{K}$ and $\nabla \Jade(v^{(k)})$ is the Fréchet gradient of $\Jade$, representing the Fréchet differential $J'_{\delta,\varepsilon}(v^{(k)})$. This approach, which allows to avoid subroutines, has also been investigated in \cite{blank}, where it has been applied to a linear elasticity problem. In our case, preliminary tests have not shown a significant discrepancy between the usage of PDAS or of \eqref{eq:projGrad}, therefore we make use of the former.
\end{remark}

\subsection{An iterative algorithm for the regularized problem}

In this subsection, we tackle the minimization of the functional $J$, namely, the stable reconstruction of cavities via perimeter-based regularization. The core idea of our approach is to iteratively apply Algorithm \ref{alg:Jade} for decreasing values of $\epsilon$ and $\delta$, using the stationary point of the previous $J_{\delta_n,\varepsilon_n}$ as a starting point for the minimization of the new functional.

\begin{algorithm}{Reconstruction of critical points of $J$} \label{alg:J}
\begin{enumerate}
    \item select initial values $(\varepsilon_0,\delta_0)$
    \item start from an initial guess $v_h^{(0,0)}$
    \item for $n = 0, 1, \ldots$
        \begin{itemize}
            \item apply Algorithm \ref{alg:Jade} on $v_h^{(0,n)}$ until convergence to $v_h^{(*,n)}$
\item update $(\varepsilon_{n+1},\delta_{n+1})$
\item set $v_h^{(0,n+1)} = v_h^{(*,n)}$.
        \end{itemize}
\end{enumerate}
\end{algorithm}

Despite it is impossible to prove the convergence of the iterates to a minimizer of $J$, such an algorithm is motivated by several considerations. \\
Firstly, for large values of $\varepsilon$, we conjecture that the (discrete version of the) functional $\Jade$ is convex, as it has been proved in \cite[Theorem 3.1]{burger2006phase} for a linear elasticity problem. In this case, Algorithm \ref{alg:Jade} is expected to converge to a global minimum of $\Jade$ (see, e.g. \cite[Corollary 27.10]{bauschke2011convex} regarding the projected gradient scheme). This is also supported by numerical evidence: the first steps of Algorithm \ref{alg:J} are done efficiently, and provide a good starting point for the ones with smaller $\varepsilon$.
\\
Secondly, the sequence of stationary points $v_h^{(*,n)}$ generated by Algorithm \ref{alg:J} is supposed to converge to a stationary point of $J$. Notice that this is not guaranteed by the $\Gamma$-convergence of the functional $\Jade$ to $J$ (separately in $\varepsilon$ and in $\delta$), because we cannot guarantee that $v_h^{(*,n)}$ are minimizers. Nevertheless, e.g. \cite[Theorem 8.1]{braides2014local} shows how to define a minimizing movement for a limit functional starting from a minimizing movent along a sequence of functionals, and in particular \cite[Chapter 11]{braides2014local} and \cite{sternberg2009critical} provides conditions under which a sequence of critical points converge to a stationary point for $J$. Unfortunately, the functionals $\Jade$ and $J$ do not match the required assumptions: in particular, $J$ is not convex and not differentiable, due to its definition on the non-convex space $X_{0,1}$.

\section{Numerical experiments}
\label{sec:results}

In this section, after a brief resume of the numerical setting for the simulations, we report and comment the results of our numerical experiments. We first analyze the performance of Algorithm \ref{alg:Jade}, particularly focusing on the dependence of the solution on the choice of the parameters $\varepsilon$ and $\delta$. Then, we move to the study of Algorithm \ref{alg:J}, of which we assess the effectiveness even on complicated shapes, and the robustness with respect to noisy data.

\subsection{Setup}

In all the experiments, we consider $\Omega \subset \R^2$ the unitary ball centered in the origin and assume to have access to the measurement on $\Gamma = \partial \Omega$. Both Algorithm \ref{alg:Jade} and \ref{alg:J} are tested making use of synthetic data: in particular, the boundary datum $u_{\operatorname{meas}}$ is generated by solving the forward problem in the presence of the true inclusion, and perturbing it with some additive Gaussian noise. To do so, we need to create an alternative mesh $\mathcal{T}_h^{\operatorname{ex}}$ of the domain $\Omega$ in the presence of the exact cavity and the associated finite element space $V_h^{\operatorname{ex}}$.
The value of the solution at the external boundary is then interpolated on the boundary of the mesh $\mathcal{T}_h$ which is used for the reconstruction, and which does not contain any hole. Notice that this whole procedure also prevents the presence of an \textit{inverse crime}, which occurs whenever the exact data are simulated via the same model that is employed by the reconstruction algorithm.
\par
Moreover, we perform reconstructions from multiple measurements: namely, we assume that $N_{\operatorname{meas}}$ measurements $u_{meas}^i$ are available, respectively associated to different sources $f^i$ in \eqref{probcav}. In the expressions of $\Jade$ (and analogously for $J$), the data mismatch term is thus replaced by an average of the mismatch of every $u_\delta^i(v)$ with respect to $u_{meas}^i$, where $u_\delta^i(v)$ is the solution of \eqref{eq:probclassic} in the presence of a cavity $v$ and with forcing term $f^i$. In order to comply with Assumption \ref{as:4} on $f$, we consider $N_{\operatorname{meas}} = 4$ measurements associated with the sources
\[ 
f^i(x,y) = \text{exp}\left\{-\frac{(x-x_i)^2+(y-y_i)^2}{r_f^2}\right\}; \quad (x_i,y_i) = R_f\left(\cos\left(\frac{i\pi}{N}\right),\sin\left(\frac{i\pi}{N_{\operatorname{meas}}}\right)\right),
\]
which are well localized in space close to the points $(x_i,y_i)$, that are sufficiently close to the boundary (and far from the cavity) depending on $R_f$.
Every datum $u_{meas}^i$ is generated by considering the traces of the exact solution associated to $f^i$ and adding random noise with Gaussian distribution with null mean and standard deviation equal to $\eta_{\text{noise}}= \eta_{\operatorname{noise}}\max\{u^i(x): x \in \Gamma \}$. Whenever not specified, we consider a $1\%$ noise level: $\eta_{\operatorname{noise}}=0.01$. 
\par
All computations are implemented with Matlab R2021a, running on a laptop with 16GB RAM and 2.2GHz CPU. We acknowledge the use of the MATLAB redbKIT library \cite{redbKIT} for the implementation of the finite element assemblers.

\subsection{Algorithm \ref{alg:Jade}: numerical results}

As depicted in section \ref{sec:numerics}, Algorithm \ref{alg:Jade} is a more efficient version of the one proposed in \cite{BRV}, to which we refer for a complete numerical analysis. In the current study, we are mostly interested in reporting the behavior of the reconstructed solution with respect to $\varepsilon$ and $\delta$.

The phase-field parameter $\varepsilon$ is strictly connected with the so-called diffuse interface region. Indeed, the minimizer of $\Jade$ is expected to be different from $\{0,1\}$ only in a small region, typically corresponding to a tubular neighborhood of the boundary of a reconstructed cavity, whose width is proportional to $\varepsilon$. A small value of $\varepsilon$ is thus preferred, but requires a sufficient refinement of the mesh, which is attained without affecting the efficiency of the algorithm by means of a local adaptive refinement. 
In Figure \ref{fig:eps} we set $\delta = 10^{-5}$ and compare the reconstructions associated with different values of $\delta$, ranging from $0.0125$ to $0.05$. In each graphic, we report a contour plot of the reconstructed indicator function, together with a dashed line denoting the boundary of the exact inclusion. It is possible to notice the dependence of thickness of the diffusion interface region from $\varepsilon$.

\begin{figure}
\subfloat[$\varepsilon = 0.05$]{\includegraphics[width=0.33\textwidth]{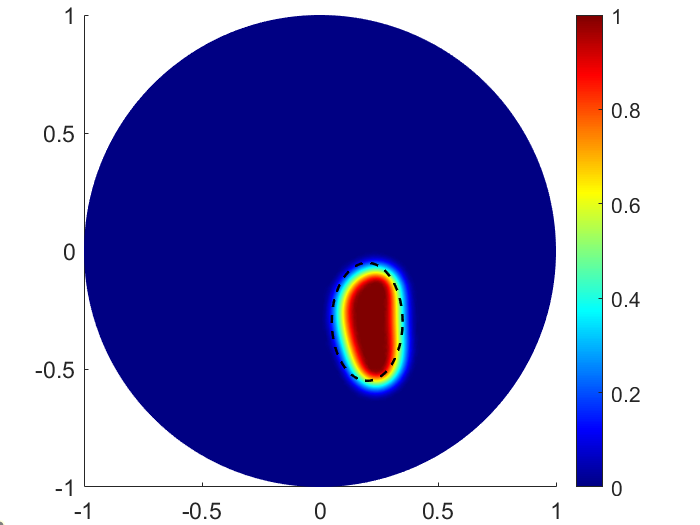}}   
\subfloat[$\varepsilon = 0.025$]{\includegraphics[width=0.33\textwidth]{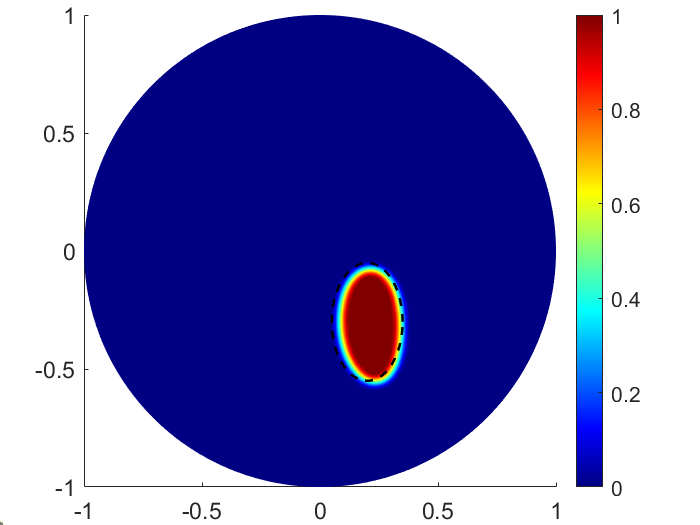}}
\subfloat[$\varepsilon = 0.0125$]{\includegraphics[width=0.33\textwidth]{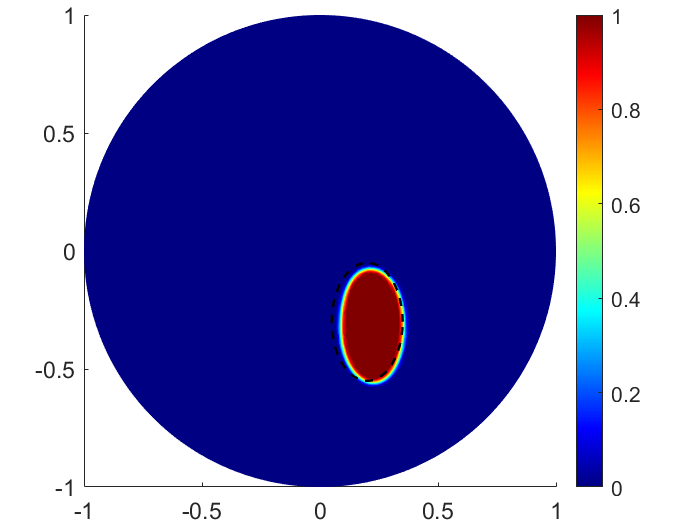}}
\caption{Dependence of the reconstruction on $\varepsilon$}
    \label{fig:eps}
\end{figure}

The fictitious conductivity $\delta$ can be chosen independently of $\varepsilon$. If $\delta$ is close to $1$, equation \eqref{eq:probclassic} becomes significantly different from a cavity problem \eqref{probcav}, thus the reconstruction is expected to be less accurate; whereas for much smaller values $\delta$ the forward problem becomes numerically unstable. In particular, it is easy to show that the $H^1$ norm of $u_{\delta,h}(v)$ is bounded by a term scaling as $\frac{1}{\delta}$. Also in this case, nevertheless, a local refinement in the region where the gradient of $v_h^{(k)}$ is steep is beneficial to reduce the ill-conditioning of the problem. In Figure \ref{fig:delta} we set $\varepsilon = 0.025$ and compare the reconstructions associated to different values of $\delta$, ranging from $10^{-5}$ to $10^{-3}$.
\begin{figure}
\subfloat[$\delta = 10^{-3}$]{\includegraphics[width=0.33\textwidth]{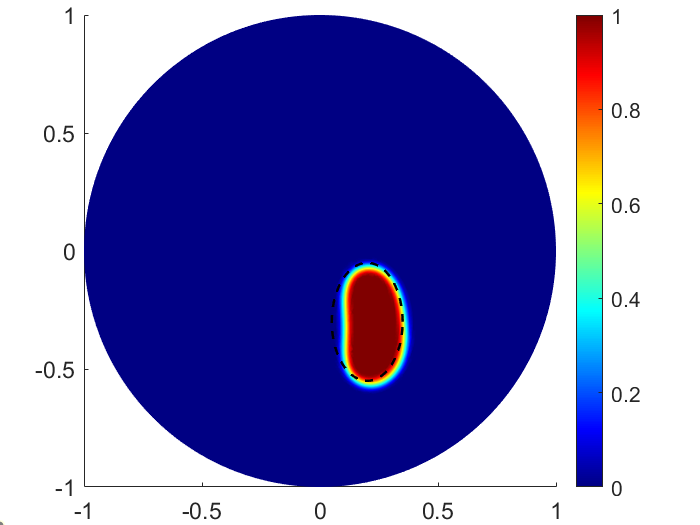}} 
\subfloat[$\delta = 10^{-4}$]{\includegraphics[width=0.33\textwidth]{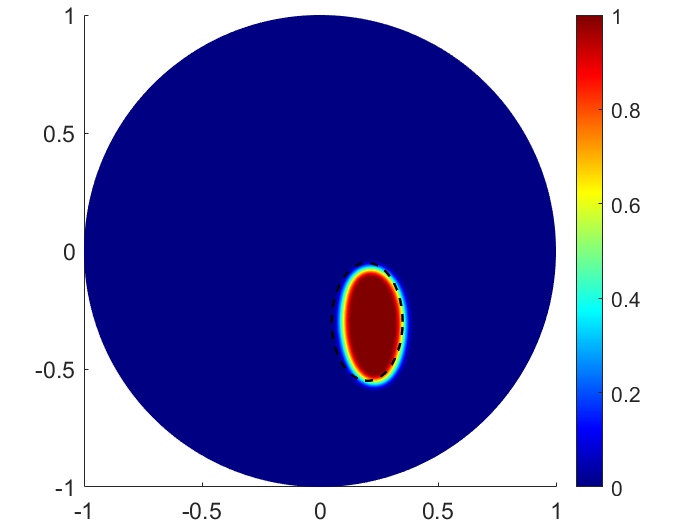}}
\subfloat[$\delta = 10^{-5}$]{\includegraphics[width=0.33\textwidth]{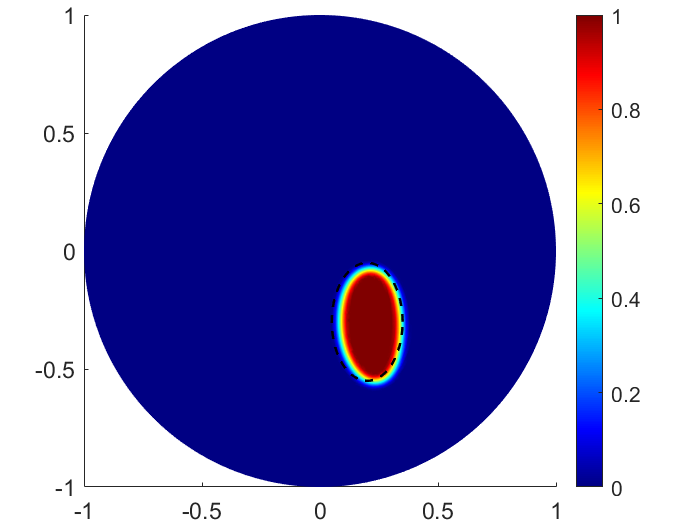}}
\caption{Dependence of the reconstruction on $\delta$}
    \label{fig:delta}
\end{figure}

In all the proposed examples, the regularization parameter $\alpha$ is chosen heuristically, and we use as a stopping criterion the relative distance between the iterates. The number of iterations required to reach convergence ranges between $100$ and $200$, which consists in a significant speedup with respect to the case without the step adaptation, which often requires over a thousand iterations (see \cite{BRV})

\subsection{Algorithm \ref{alg:J}: numerical results}

In the numerical implementation of Algorithm \ref{alg:J}, we initialize $\varepsilon$, $\delta$ by $\varepsilon_0 = 0.1$ and $\delta_0 = 10^{-2}$, and reduce them by a factor $4$ and $10$, respectively. The initial guess for $v_h^{(0,0)}$ is a constant function of value $0$.
In Figures \ref{fig:combined2} we show some results of the application of the combined algorithm for the minimization of $\Jade$ for the reconstruction of a polygonal cavity. 

\begin{figure}
\subfloat[Iteration 60]{\includegraphics[width=0.33\textwidth]{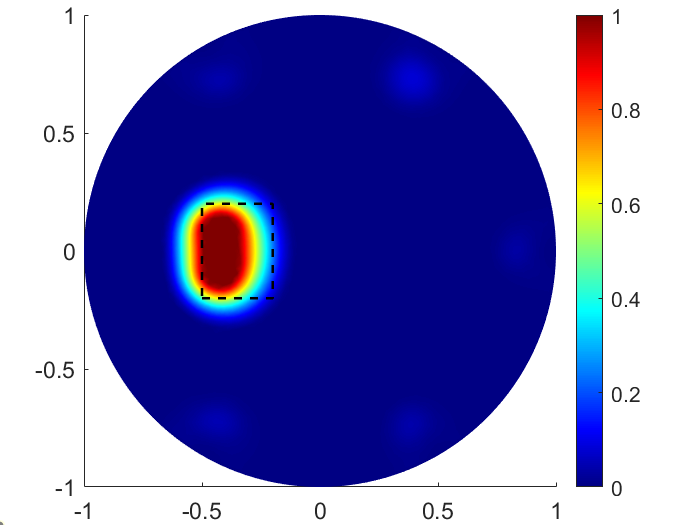}} 
\subfloat[Iteration 120]{\includegraphics[width=0.33\textwidth]{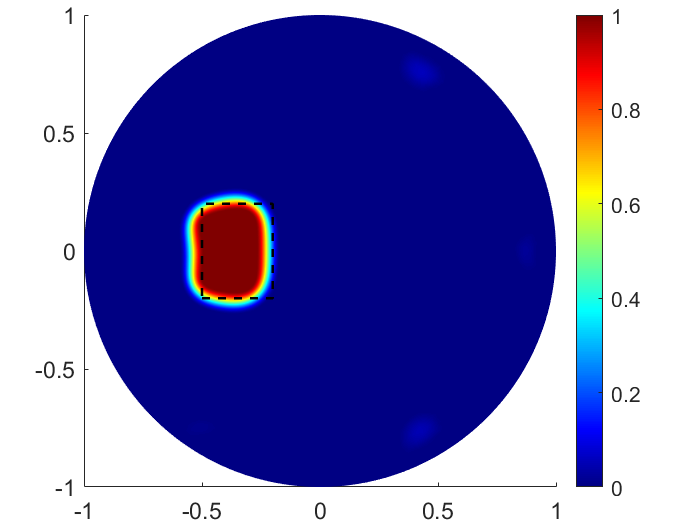}}
\subfloat[Iteration 180]{\includegraphics[width=0.33\textwidth]{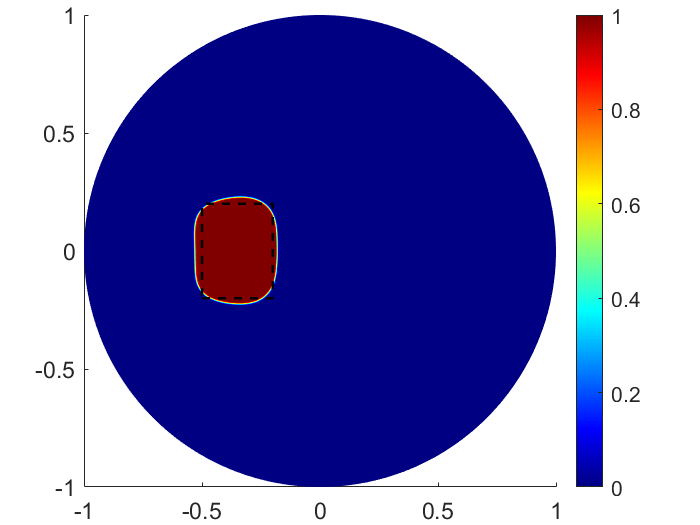}}
\caption{Combined algorithm: results}
    \label{fig:combined2}
\end{figure}

In Figure \ref{fig:nonconv} we report some additional results showing that the algorithm can effectively tackle the reconstruction of more complicated domains, such as non-convex ones and ones consisting of more than a single connected component. 
\begin{figure}
\subfloat[]{\includegraphics[width=0.33\textwidth]{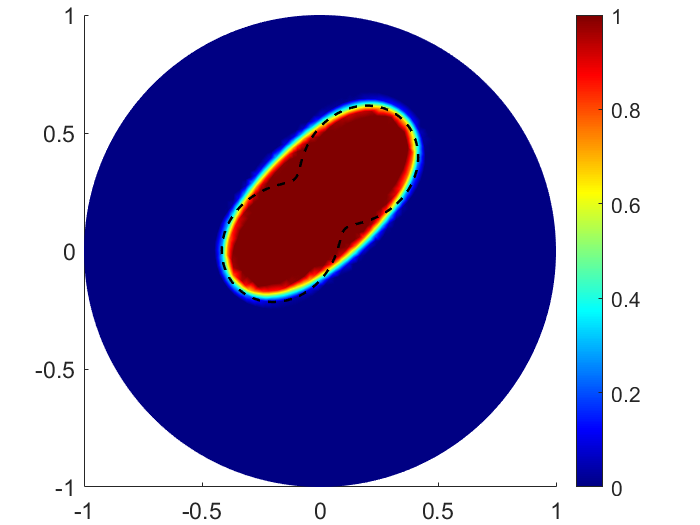}}
\subfloat[]{\includegraphics[width=0.33\textwidth]{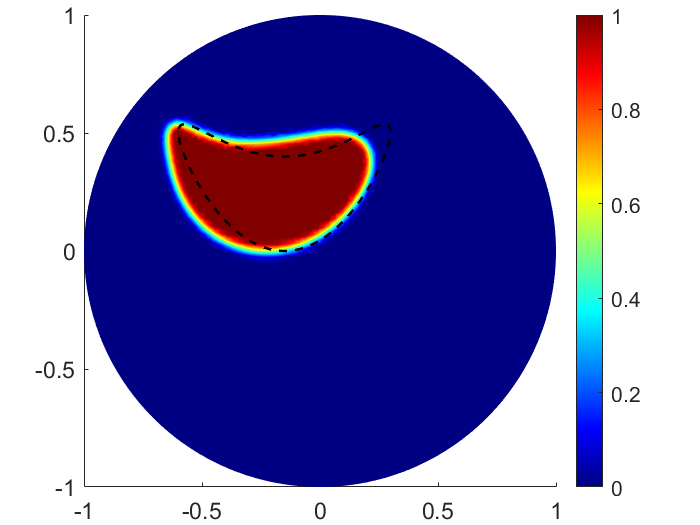}}
\subfloat[]{\includegraphics[width=0.33\textwidth]{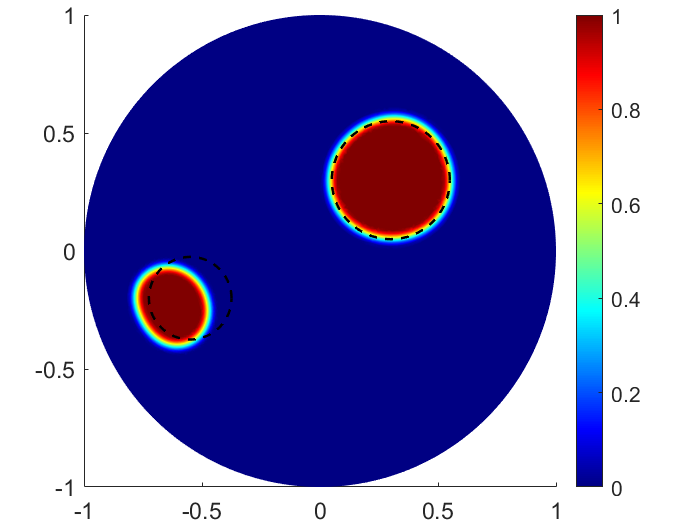}}
\caption{Additional reconstructions: non-convex domains}
    \label{fig:nonconv}
\end{figure}

As a final study, we discuss the behavior of the proposed algorithm in the presence of higher noise level. As previously explained, all the simulations analyzed so far are based on synthetic data perturbed by a Gaussian noise with variance equal to the $1\%$ of the peak value of the signal. In Figure \ref{fig:noise}, we report the reconstructions associated with the same inclusion, but with larger level of noise ((a): $2\%$, (b): $5\%$). As depicted in (c), a higher level of noise can be treated by increasing the value of the regularization parameter $\alpha$, at the price of a lower quality of the reconstruction.
\begin{figure}
\subfloat[$\eta_{\operatorname{noise}} = 0.02, \alpha = 10^{-5}$ ]{\includegraphics[width=0.33\textwidth]{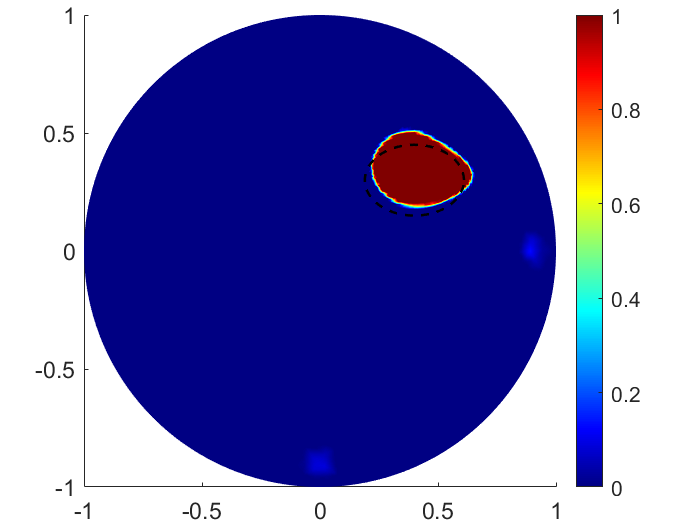}}
\subfloat[$\eta_{\operatorname{noise}} = 0.05, \alpha = 10^{-5}$ ]{\includegraphics[width=0.33\textwidth]{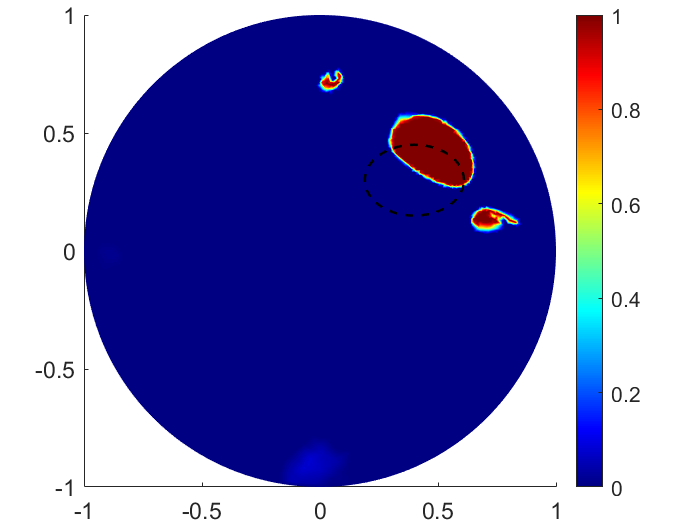}}
\subfloat[$\eta_{\operatorname{noise}} = 0.05, \alpha = 10^{-4}$ ]{\includegraphics[width=0.33\textwidth]{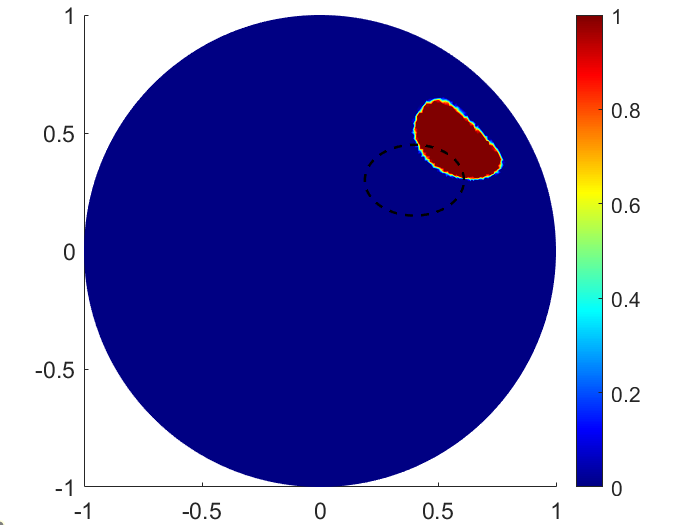}}
\caption{Reconstruction in the presence of large noise}
    \label{fig:noise}
\end{figure}

\section{Final remarks}
We have analyzed the problem of reconstructing Lipschitz cavities from boundary measurements in a model arising from cardiac electrophysiology. 
The reconstruction algorithm relies on a detailed investigation of the dependence of the solutions to the direct problem on the cavities and is based on a phase-field approach that we justify via $\Gamma-$ convergence of a relaxed family of functionals
$J_{\varepsilon,\delta}$ to the original penalized misfit functional $J$. This implies convergence of minima of $J{\varepsilon,\delta}$ to minima of $J$.  

In order to prove our result we have to restrict the relaxed functionals to a non convex subset $\mathcal{K}_{\eta}$ of the convex set of admissible functions $\mathcal{K}$, while in the numerical algorithm we need to minimize the approximating functionals $J_{\varepsilon,\delta}$ over the whole convex set $\mathcal{K}$; nevertheless, as 
discussed in remark \ref{convrelax}, numerical calculations seem to indicate that the minima of the functional  $J_{\delta,\epsilon}$ in $\mathcal{K}$ belong to $\mathcal{K}_{\eta}$. 

Although we have not found a theoretical justification to this property, it could be useful to remark that the $\Gamma-$ convergence of $J_{\delta,\epsilon}$ to $J_{\delta}$ (theorem \ref{thm:Gconv_eps}) and the resulting convergence of the minima (corollary \ref{convmin}) may also be achieved on different subsets $\mathcal{H}\subseteq \mathcal{K}$. 

In fact, by inspection of the proof of the above results, one finds that $\mathcal{H}$ should be a weakly closed subset of $H^1(\Omega)$
such that:

\begin{itemize}
 \item if $v_n\in\mathcal{H}$ is such that $v_n\rightarrow \chi_{\Omega_D}$ in $L^1(\Omega)$, then $D\in {\mathcal D}$, where ${\mathcal D}$ was defined in Assumption $3$;
    \item $\mathcal{H}$ contains the functions $v_{\varepsilon,\beta}$ defined in the proof of theorem \ref{thm:Gconv_eps}) for some $\eta>0$.
\end{itemize}

Note that the first condition is needed in the proof of the lim inf property and the last one for the lim sup property.
It is not clear to us if it is possible to construct a subset $\mathcal{H}$ which is also convex (this would somehow justify the 'convex relaxation' argument of remark \ref{convrelax}).

\section{Acknowledgements}
We would like to thank Giovanni Bellettini for the stimulating and useful suggestions.
The work of LR is supported by the Air Force Office of Scientific Research under award number FA8655-20-1-7027. 
The authors are members of the ``Gruppo Nazionale per l'Analisi Matematica, la Probabilit\`a e le loro Applicazioni'' (GNAMPA), of the ``Istituto Nazionale per l'Alta Matematica'' (INdAM).

 \section{Appendix}

In this appendix we prove Proposition \ref{convudeltan} where we make use of a suitable version of a Caccioppoli type inequality which we prove below.
 \begin{theorem}
 (Caccioppoli type)
 Let $A=A(x)$, $x\in B_{2R}$ be a symmetric $n\times n$ matrix, $L^\infty$ and elliptic.
 Let $w\in L^\infty(B_{2R})$ be a weight such that $0<\delta\leq w\leq1$ a.e. in $B_{2R}$ and $u$ a solution in a weak sense to

 $$-\div(wA\nabla u) +wu^3.$$

 Then
\begin{equation}
 \label{eq:caccioppoli2}   
 \int_{B_R}w\,|\nabla u|^2 \leq \frac{\tilde{C}}{R^2}\int_{B_{2R}}w\;u^2=0 \ \ \ \rm{in} \ \ B_{2R}
 \end{equation}

 \end{theorem}

 \begin{proof}
  Let $\chi\in \mathcal{C}_0^\infty(B_{2R})$, $0\leq\chi\leq1$ in $B_{2R} $, $\chi\equiv1$ in $B_{R}$ and  $|\nabla\chi|\leq\displaystyle{\frac{C}{R}}$ on $B_{2R}$.

  In the weak formulation take as test function $\varphi=u\chi^2$

 $$-\int_{B_{2R}}div (wA\nabla u)u\chi^2 +\int_{2B}wu^4\chi^2=0$$

i.e., integrating by parts,

$$-\int_{\partial B_{2R}}wA\nabla u \cdot\nu u\chi^2 +\int_{B_{2R}}w\nabla u\,A\nabla(u\chi^2)+\int_{B_{2R}}wu^4\chi^2=0$$

where the first term is $=0$ because ot the definition of $\chi$. Then we have

 \[
\int_{B_{2R}} w\nabla u A\nabla u \chi^2+\int_{B_{2R}} w \nabla u\,A\,\nabla\chi^2 u+\int_{B_{2R}}w\,u^4\chi^2=0
 \]

using ellipticity condition and boundedness of $A$ and of throwing away the last term (that is $\geq0$) , we get

\[
\lambda\int_{B_{2R}} w|\nabla u|^2 \chi^2\leq \int_{B_{2R}}2 w|\nabla u A\nabla \chi|\;|u|\;\chi\leq2\Lambda\int_{B_{2R}} w|\nabla u|\;|\nabla \chi|\;|u|\chi
\]
 and using Young's inequality the latest is

 \[
 \leq \epsilon \int_{B_{2R}}w \;|\nabla u|^2\;\chi^2 + \frac{2\Lambda}{\epsilon}\int_{B_{2R}}w\;|u|^2\,|\nabla\chi|^2
 \]

from which, using the properties of $\chi$,

\[
(\lambda-\epsilon)\int_{B_R} w\,|\nabla u|^2\leq(\lambda-\epsilon)\int_{B_{2R}}w\,|\nabla u|^2\chi^2\leq \frac{2\Lambda C^2}{\epsilon \;R^2}\int_{B_{2R}}w\;|u|^2
\]

Taking for instance $\epsilon=\frac{\lambda}{2}$ the theorem is proved with $\displaystyle{\tilde{C}=\frac{8\Lambda \;C^2}{\lambda^2}}.$
 \end{proof}
 \bigskip
 
{\bf Proof of Proposition \ref{convudeltan}}\\
  \begin{proof}
  For seek of clarity we divide the proof in several steps.\\
{\it First step.} We start proving some weak convergence results. Let $\{v_{\delta_n}\}_{n\geq 1}$ be a given sequence of elements in $X_{0,1}$ converging in $L^1(\Omega)$ as $\delta_n\rightarrow 0$ to an element $v$.Then by Lemma \ref{closeX} it follows that $v\in X_{0,1}$ i.e. $v=\chi_D$ with $D\in\mathcal{D}$.  
 Thus, $v_{\delta_n}\rightarrow v:=\chi_{\Omega_D}$ in $L^1(\Omega)$,  a.e. in $\Omega$ and also,
\begin{equation}\label{adeltanconv}
\sqrt{a_{\delta_n}}\rightarrow \chi_{\Omega_D},\,\, a_{\delta_n}:=a_{\delta_n}(v_n) = \delta_n + (1-\delta_n) v_n\rightarrow \chi_{\Omega_D},\text{ a.e. in }\Omega \text{ and in } L^p(\Omega)
\end{equation}
for any $p\in [1,\infty]$.
Consider now $u_{\delta_n}$ solution of Problem  \ref{eq:probclassic} for $v=v_{\delta_n}$.
Then from \eqref{estimateudelta}, \eqref{estimategradudelta}, we know that the sequences $\{ \sqrt{a_n} u_n \}$ and $\{ \sqrt{a_n} \nabla u_n \}$ are uniformly bounded, respectively in $L^2(\Omega)$ and in $L^2(\Omega; \R^2)$; so, possibly up to a subsequence,
\begin{align}
\sqrt{a_{\delta_n}} u_{\delta_n} &\rightharpoonup \tilde{u} \text{ in } L^2(\Omega)    \label{eq:weaku} \\
\sqrt{a_{\delta_n}} \nabla u_{\delta_n} &\rightharpoonup V \text{ in } L^2(\Omega; \R^2)    \label{eq:weakgradu}
\end{align}

{\it Second step.} In this part we will show that the weak limits $\tilde{u}$ and of $V$ are a.e. equal to zero inside the cavity $D$. In fact,  observe that for any $B_{2R} (y)\subset D$ we have $a_{\delta_n}\rightarrow0$ a.e. in $B_{2R} (y)$ and by dominated convergence theorem

\begin{equation}\label{B2R0}
\int_{B_{2R}}  a_{\delta_n}\;u_{\delta_n}^2 \rightarrow 0,
\end{equation}
and so $$\sqrt{a_{\delta_n}}\;u_{\delta_n}\rightarrow 0$$ in $L^2(B_{2R}(y)$. By uniqueness of the (weak) limit, from (\ref{eq:weaku}), we deduce that 
$
\tilde{u}\equiv 0,\,\ \text{a.e. in } B_{2R}(y)
$
and by the arbitrariness of $y$ it follows
\begin{equation}\label{uincavity}
\tilde{u}\equiv 0,\,\ \text{a.e. in } D.
\end{equation}
In order to conclude a similar result for $V$  we apply the Cacciopoli type inequality \eqref{eq:caccioppoli2}
\begin{equation}
\int_{B_R} a_{\delta_n} |\nabla u_{\delta_n}|^2 \leq C \int_{B_{2R}} a_{\delta_n} u_{\delta_n}^2.
    \label{eq:Caccioppoli}
\end{equation}
which entails
\[
\sqrt{a_{\delta_n}}\;\nabla u_{\delta_n}\rightarrow0\; \text{in}\;L^2\left(B_R(y),\R^2\right)
\]
and which implies 
\begin{equation}\label{Vincavity}
V\equiv\vec{0},\,\ \text{a.e.} \ \ D
\end{equation}
Hence, by the fact that $a_{\delta_n}\rightarrow \chi_D$ a.e. in $\Omega$ and by (\ref{eq:weaku}) and \eqref{eq:weakgradu})  we also have that 
\begin{align}
a_{\delta_n} u_{\delta_n} &\rightharpoonup \tilde{u} \text{ in } L^2(\Omega)    \label{eq:weak_u2} \\
a_{\delta_n} \nabla u_{\delta_n} &\rightharpoonup V \text{ in } L^2(\Omega; \R^2)    \label{eq:weak_gradu2}
\end{align}

 {\it Third step.} In this part of the proof we will show that $\tilde{u}\in H^1(\Omega_D)$ and that $\nabla\tilde{u}=V$ a.e. in $\Omega_D$. Fix $a>0$ and define the set $D^a\equiv\{x\in \Omega \,|\, dist(x,D)\leq a\}$ and let $N_0:=N_0(a)$ be such that for $n\geq N_0$ $D_{n}\subset D^a$ and $a_{\delta_n}=1$ in $\Omega_{ D^a}$. Then by (\ref{estimateudelta}) and (\ref{estimategradudelta}) the following uniform estimate holds
  \begin{equation}\label{estimatesinDa}
 \|u_{\delta_n}\|_{H^1(\Omega_{ D^a})}\leq C
 \end{equation}
 which implies that, possibly up to a subsequence, that for some $U^a\in H^1(\Omega_{D^a})$  
 \begin{equation}\label{weakH1convtoU}
 u_{\delta_n}\rightharpoonup U^a \textrm{ in } H^1(\Omega_{D^a})
 \end{equation}
 and therefore strongly in $L^2(\Omega_{D^a})$  
\begin{equation}\label{strongconv}
u_{\delta_n}\rightarrow U^a\textrm{ in } L^2(\Omega_{D^a}).
\end{equation}
Now, from \ref{eq:weaku} and recalling that $a_{\delta_n}=1$ in $\Omega_{ D^a}$ for $n>N_0$ we can infer that
\begin{equation}
\int_{\Omega}u_{\delta_n}\varphi  \rightarrow\int_{\Omega}U^a\varphi 
\end{equation}
 for any  $\varphi\in L^2(\Omega)$ such that $\varphi=0$ in $\Omega\backslash\Omega_{ D^a}$. So, 
 $$
 \int_{\Omega_{ D^a}}U^a\varphi=\int_{\Omega_{ D^a}}\tilde{u}\varphi,\,\,\forall\varphi\in L^2(\Omega)
 $$
 which implies that
 $$
 U^a=\tilde{u}|_{\Omega_{ D^a}}
 $$
 and hence 
 \begin{equation}\label{gradequality}
 \nabla(U^a)=\nabla(\tilde{u}|_{\Omega_{ D^a}})\textrm{ in }L^2(\Omega_{D^a},\mathbb{R}^2).
 \end{equation}
 Let now $\varphi\in H^1(\Omega)$ (observe that $\varphi\in H^1(\Omega_{ D^a})\, \forall a>0$). From \ref{eq:weaku} and \ref{eq:weakgradu} it follows that for $n>N_0$
 $$
 \int_{\Omega}(\sqrt{a_{\delta_n}} \nabla u_{\delta_n}\cdot\nabla\varphi+\sqrt{a_{\delta_n}} u_{\delta_n}\varphi)
 =\int_{\Omega_{ D^a}\cup (D_a\backslash D)\cup (D\backslash D^a)}(\sqrt{a_{\delta_n}} \nabla u_{\delta_n}\cdot\nabla\varphi+\sqrt{a_{\delta_n}} u_{\delta_n}\varphi)
 $$
 Hence, setting
 \[
 \int_{D^a\backslash D}(\sqrt{a_{\delta_n}} \nabla u_{\delta_n}\cdot\nabla\varphi+\sqrt{a_{\delta_n}} u_{\delta_n}\varphi)=\epsilon_n(a)
 \]
 and observing that by (\ref{eq:weaku}), (\ref{eq:weakgradu}) and (\ref{uincavity}), (\ref{Vincavity}) one has that
 \[
 \int_{D\backslash D^a}(\sqrt{a_{\delta_n}} \nabla u_{\delta_n}\cdot\nabla\varphi+\sqrt{a_{\delta_n}} u_{\delta_n}\varphi)=o(1).
 \]
 we can write
 $$
 \int_{\Omega}(\sqrt{a_{\delta_n}} \nabla u_{\delta_n}\cdot\nabla\varphi+\sqrt{a_{\delta_n}} u_{\delta_n}\varphi)
 =\int_{\Omega_{ D^a}}(\sqrt{a_{\delta_n}} \nabla u_{\delta_n}\cdot\nabla\varphi+\sqrt{a_{\delta_n}} u_{\delta_n}\varphi)+\epsilon_n(a)+o(1).
 $$
 Then again by (\ref{eq:weaku}) and (\ref{eq:weakgradu}) we can write 
 \begin{equation}\label{relationtildeu1}
 \int_{\Omega}(V\cdot\nabla\varphi+\tilde{u}\varphi)+o(1)
 =\int_{\Omega_{ D^a}}( \nabla(\tilde{u}|_{\Omega_{ D^a}})\cdot\nabla\varphi+ \tilde{u}\varphi)+\epsilon_n(a).
 \end{equation}
 and by (\ref{uincavity}) and (\ref{Vincavity}) this last relation also implies that
 \begin{equation}\label{relationtildeu2}
 \int_{\Omega_D}(V\cdot\nabla\varphi+\tilde{u}\varphi)+o(1)
 =\int_{\Omega_{ D^a}}( \nabla(\tilde{u}|_{\Omega_{ D^a}})\cdot\nabla\varphi+ \tilde{u}\varphi)+\epsilon_n(a).
 \end{equation}
 Finally, let us pick up $a=a_n\rightarrow 0$ as $n\rightarrow \infty$ in (\ref{relationtildeu2}) and consider a sequence $\tilde{u}_n\in H^1(\Omega_D)$ such that $\tilde{u}_n|_{\Omega_{D^{a_n}}}=\tilde{u}$ and with $\tilde{u}_n\rightarrow \tilde{u}$ in $L^2(\Omega_D)$. Then the integral on the right-hand side of (\ref{relationtildeu2}) can be written in the form
 \[
 \int_{\Omega_{ D^{a_n}}}( \nabla\tilde{u}_n\cdot\nabla\varphi+ \tilde{u}_n\varphi)=\int_{\Omega_D}( \nabla(\tilde{u}_n)\cdot\nabla\varphi+ \tilde{u}_n\varphi)+\tilde{\epsilon}_n(a_n).
 \]
 We observe that by using Schwartz inequality (and the uniform estimate \eqref{estimatesinDa} we have that $\epsilon_n(a_n),\tilde{\epsilon}_n(a_n)$ both converge to zero as $n\rightarrow \infty$. Hence,  
 \[
  \int_{\Omega_D}(V\cdot\nabla\varphi+\tilde{u}\varphi)=\lim_{n\rightarrow \infty}
  \int_{\Omega_D}( \nabla \tilde{u}_n\cdot\nabla\varphi+ \tilde{u}_n\varphi)
 \]
i.e. $\tilde{u}_n\rightharpoonup\tilde{u}$ in $H^1(\Omega_D)$ and $\nabla\tilde{u}=V$.\\
{\it Fourth step.} We now show that $\tilde{u}$ is the solution of the cavity problem \eqref{probcav} i.e.
 \\
\begin{equation}\label{equtil}
\int_{\Omega_D} \nabla \tilde{u} \cdot\nabla \varphi + \int_{\Omega_D} \tilde{u}^3\varphi = \int_{\Omega_D} f\;\varphi  \ \ \ \ \  \textrm{for} \ \ \varphi\in H^1(\Omega_D)
\end{equation}

Consider the weak formulations for $u_{\delta_n}$

\begin{equation}\label{equdeltan}
\int_\Omega a_{\delta_n}\nabla u_{\delta_n}\cdot \nabla \tilde{\varphi} + \int_\Omega a_{\delta_n} u_{\delta_n}^3\tilde{\varphi} = \int_\Omega f\;\tilde{\varphi } \ \ \ \ \  \textrm{for} \ \ \tilde{\varphi}\in H^1(\Omega)
\end{equation}

Consider $\varphi\in H^1(\Omega_D)$ and extend it to $\tilde{\varphi}\in H^1(\Omega)$. Then subtracting \eqref{equtil} from \eqref{equdeltan} we obtain

\begin{equation}\label{equdiffer}
\int_{\Omega_D} (a_{\delta_n}\nabla u_{\delta_n} - \nabla\tilde{u} )\cdot\nabla {\varphi} + \int_{\Omega_D} (a_{\delta_n} u_{\delta_n}^3-\tilde{u} ^3){\varphi}+\int_D a_{\delta_n}\nabla u_{\delta_n}\cdot \nabla \tilde{\varphi} + \int_D a_{\delta_n} u_{\delta_n}^3\tilde{\varphi}  = 0
\end{equation}

Because of the convergence results collected in the previous steps  all the terms in \eqref{equdiffer} tend to $0$ as $\delta_n\rightarrow0$.

{\it Step 5. } Let us finally prove the convergence of the traces in $L^2$ i.e.
  $$\|u_{\delta_n}-\tilde{u}\|_{L^2(\Sigma)}\rightarrow 0.$$

 In \eqref{equtil} and \eqref{equdeltan} take the test function $\varphi=(u_{\delta_n}-\tilde{u})\chi^2\in H^1(\omd)$ where $\chi$ is the cutoff function such that $0\leq\chi\leq1$ in $\omd$, $\chi=1$ in $\Omega_{d_0/2}$, $\chi=0$ in $\omd\setminus\Omega_{d_0}$ and $|\nabla\chi|\leq\frac{C}{d_0}$.

 Plugging $\varphi$ into \eqref{equdeltan} and in \eqref{equtil} and subtracting the two equations, recalling that $supp (f)\subset\Omega_{d_0}$,  we obtain

 \begin{equation}\label{eqdiff}
 \int_{\Omega_{d_0}}\nabla(u_{\delta_n}-\tilde{u})\cdot\nabla[(u_{\delta_n}-\tilde{u})\chi^2]+\int_{\Omega_{d_0}}(u_{\delta_n}^3-\tilde{u}^3)(u_{\delta_n}-\tilde{u})\chi^2=0
 \end{equation}

 i.e.

 \begin{equation}\label{eqdiff2}
 \int_{\Omega_{d_0}}|\nabla(u_{\delta_n}-\tilde{u})|^2\chi^2+2\int_{\Omega_{d_0}}\nabla(u_{\delta_n}-\tilde{u})\nabla\chi(u_{\delta_n}-\tilde{u})\chi+\int_{\Omega_{d_0}}(u_{\delta_n}-\tilde{u})^2(u_{\delta_n}^2+u_{\delta_0}\tilde{u}+\tilde{u}^2)\chi^2=0
 \end{equation}

 Applying Young's inequality to the second term in \eqref{eqdiff2} we obtain that

 \begin{equation}\label{eqdiff3}
 \left|\int_{\Omega_{d_0}}\nabla(u_{\delta_n}-\tilde{u})\cdot\nabla\chi(u_{\delta_n}-\tilde{u})\chi\right|\leq \epsilon\int_{\Omega_{d_0}}|\nabla(u_{\delta_n}-\tilde{u})|^2\chi^2+\frac{1}{\epsilon}\int_{\Omega_{d_0}}|\nabla\chi|^2|u_{\delta_n}-\tilde{u}|^2
 \end{equation}

 and combining the above back in \eqref{eqdiff2}, reordering terms, using the properties of $\chi$ and the $L^\infty$ estimates on $\tilde{u}$ and $u_{\delta_n}$ we get

 \[
 \int_{\Omega_{d_0}}|\nabla(u_{\delta_n}-\tilde{u})|^2\chi^2\leq C\int_{\Omega_{d_0}}|u_{\delta_n}-\tilde{u}|^2
 \]

which implies

 \[
 \int_{\Omega_{d_0/2}}|\nabla(u_{\delta_n}-\tilde{u})|^2\leq C\|u_{\delta_n}-\tilde{u}\|_{L^2(\Omega_{d_0})}^2
 \]

 and therefore,
 using the fact that $v_{\delta_n}=1$ a.e. in $\Omega_{d_0}$, the fact that $a_{\delta_n}=1$ a.e. in $\Omega_{d_0}$ and the convergences proved above, that

 $$\|u_{\delta_n}-\tilde{u}\|_{H^1(\Omega_{d_0/2})}\leq C\|u_{\delta_n}-\tilde{u}\|_{L^2(\Omega_{d_0})}\rightarrow 0.$$

 Finally, by the trace inequality we conclude that

 $$\|u_{\delta_n}-\tilde{u}\|_{L^2(\Sigma)}\rightarrow 0.$$
 concluding the proof.

 \end{proof}

\end{document}